\documentclass{amsart}
\usepackage{amsthm,stmaryrd}
\usepackage{amssymb}
\usepackage{color}
\usepackage{epsfig}
\usepackage{verbatim}
\usepackage{hyperref}
\usepackage{mathrsfs}
\usepackage[utf8]{inputenc}
\usepackage[all]{xy}
\renewcommand{\k}{{\kappa}}
\newcommand{\Gr}{{G^*}}
\newcommand{\Gd}{{G}}
\newcommand{\E}{{\ensuremath{\mathbb F}}}
\newcommand{\A}{{\mathsf A}}
\newcommand{\F}{{\mathcal F}}

\newcommand{\Ad}{\operatorname{Ad}}
\newcommand{\ev}{\operatorname{ev}}
\newcommand{\tcoev}{\wt{\operatorname{coev}}}
\newcommand{\tev}{\wt{\operatorname{ev}}}
\newcommand{\coev}{{\operatorname{coev}}}
\newcommand{\brk}[1]{{\left\langle{#1}\right\rangle}}

\newcommand{\FK}{\ensuremath{\Bbbk}}
\newcommand{\ve}{\varepsilon}
\newcommand{\ro}{r}
\newcommand{\g}{\ensuremath{\mathfrak{g}}}

\newcommand{\vp}{\varphi}
\newcommand{\kt}{$\Bbbk$\nobreakdash-\hspace{0pt}}
\newcommand{\kk}{\Bbbk}
\newcommand{\ideal}{\mathcal{I}}
\renewcommand{\t}{{\mathsf{t}}}
\newcommand{\ZUo}{{\mathcal Z^{0}}}
\newcommand{\catU}{\mathcal{D}}
\newcommand{\Proj}{\operatorname{\mathsf{Proj}}}
\newcommand{\UqHe}{\ensuremath{U_q^D}}
\newcommand{\UxiHe}{\ensuremath{U_\xi^{D}}}
\newcommand{\Uw}{\ensuremath{\wh{\UxiHe}}}
\newcommand{\UxiH}{\ensuremath{U_\xi^{H}}}
\newcommand{\Uxi}{\ensuremath{U_\xi}}
\newcommand{\FF}{\ensuremath{\mathscr F}}
\renewcommand{\F}[1]{{F^{[#1]}}} 
\newcommand{\s}{\sigma}

\newcommand{\e}{{\operatorname{e}}}
\newcommand{\slt}{{\mathfrak{sl}(2)}}
\newcommand{\SL}{{\operatorname{SL}}}

\newcommand{\unit}{\ensuremath{\mathbb{I}}}
\newcommand{\cat}{\mathscr{C}}

\newcommand{\ED}{\mathcal{E}_D}
\newcommand{\oo}{\infty}
\newcommand{\Id}{\operatorname{Id}}
\newcommand{\bp}[1]{{\left(#1\right)}}
\newcommand{\qn}[1]{{\left\{#1\right\}}}

\newcommand{\End}{\operatorname{End}}
\newcommand{\Hom}{\operatorname{Hom}}
\newcommand{\Bij}{\operatorname{Bij}}
\newcommand{\tr}{\operatorname{tr}}
\newcommand{\C}{\ensuremath{\mathbb{C}} }
\newcommand{\Z}{\ensuremath{\mathbb{Z}} }
\newcommand{\R}{\ensuremath{\mathbb{R}} }
\newcommand{\N}{\ensuremath{\mathbb{N}} }
\newcommand{\wb}{\overline}
\newcommand{\wh}{\widehat}
\newcommand{\wt}{\widetilde}
\newcommand{\wv}{\overrightarrow}
\newcommand{\wov}{\overleftarrow}
\newcommand{\ms}[1]{\mbox{\tiny$#1$}}
\newcommand{\mm}[1]{\mbox{\small$#1$}}
\newcommand{\qbin}[2]{\left[\!\!\begin{array}{c}
      #1 \\
      #2 \end{array}\!\!\right]}
\newcommand{\epsh}[2]
         {\begin{array}{c} \hspace{-1.3mm}
        \raisebox{-4pt}{\epsfig{figure=#1,height=#2}}
        \hspace{-1.9mm}\end{array}}

\newcommand{\Y}{\ensuremath{\mathsf{Y^*}}}
\newcommand{\Yd}{\ensuremath{\mathsf{Y}}}
\newcommand{\RY}{\ensuremath{\mathcal{R}}}
\newcommand{\B}{\ensuremath{{C}}}
\newcommand{\xl}{\ensuremath{\mathsf{x_L}}}
\newcommand{\xr}{\ensuremath{\mathsf{x_R}}}
\newcommand{\shift}{\ensuremath{\mathsf{s}}}
\renewcommand{\tt}{\ensuremath{\mathsf{t}}}
\newcommand{\y}{\underline y}
\renewcommand{\j}{\lambda}
\newcommand{\HHH}{H}

\newtheorem{theo}{Theorem}[section]
\newtheorem{lem}[theo]{Lemma}
\newtheorem{prop}[theo]{Proposition}
\newtheorem{cor}[theo]{Corollary}

\theoremstyle{definition}

\theoremstyle{remark}

\renewcommand{\qedsymbol}{\fbox{\thetheo}}

\newcounter{exo} \newcounter{numexercice}
\renewcommand{\theexo}{\arabic{exo}}

\begin{document}
\title{$G$-links invariants, Markov traces and the semi-cyclic
  $U_q\slt$-modules.}  

\author{Nathan Geer}
\address{Mathematics \& Statistics\\
  Utah State University \\
  Logan, Utah 84322, USA}
\thanks{Research of the first author was
  partially supported by NSF grants DMS-0968279 and DMS-1007197.}\
\email{nathan.geer@usu.edu}

\author{Bertrand Patureau-Mirand}
\address{UMR 6205, LMBA, universit\'e de Bretagne-Sud, universit\'e
  europ\'eenne de Bretagne, BP 573, 56017 Vannes, France }
\email{bertrand.patureau@univ-ubs.fr}

\begin{abstract} 
  Kashaev and Reshetikhin proposed a generalization of the
  Reshetikhin-Turaev link invariant construction to tangles with a
  flat connection in a principal $G$-bundle of the complement of the
  tangle.  The purpose of this paper is to adapt and renormalize their
  construction to define invariants of $G$-links using the semi-cyclic
  representations of the non-restricted quantum group associated to
  $\slt$, defined by De Concini and Kac.  Our construction uses a
  modified Markov trace.  In our main example, the semi-cyclic
  invariants are a natural extension of the generalized Alexander
  polynomial invariants defined by Akutsu, Deguchi, and Ohtsuki.
  Surprisingly, direct computations suggest that these invariants are
  actually equal.
\end{abstract}

\maketitle

\section{Introduction}

\subsection{} 
A major achievement in quantum topology was the construction of
invariants of links and tangles using representations of a quantum
group and more generally ribbon categories, due to Reshetikhin and
Turaev in \cite{RT}.  Turaev gave a theoretical generalization of this
construction \cite{T2000HFT,T2010HQFT} to $G$-links and $G$-tangles:
tangles with a flat connection in a principal $G$-bundle over the
complement of the tangle.  The construction is parallel to \cite{RT}
and leads to HQFT and invariants of 3-manifolds with a representation
of their fundamental group in $G$.  The algebraic data used in this
theory is a ribbon $\Gd$-category which in particular is a
$\Gd$-graded tensor category equipped with a homotopy braiding.  It is
an open problem to find interesting examples of ribbon
$\Gd$-categories where $\Gd$ is an infinite non-abelian group.

In \cite{KR05}, Kashaev and Reshetikhin propose a modification of
Turaev's construction: for a group $G$ having a factorization
$G=G_+G_-$, they introduce the group $\Gr=\Gd_+\times\Gd_-$ and define
a notion of a $\Gr$-ribbon category which is slightly different from
the previous one.  In particular, a $\Gr$-ribbon category $\cat$ is
$\Gr$-graded category which is equipped with what could be called a
\emph{holonomy braiding}.  Loosely speaking, a holonomy braiding is a
pair $(\RY, c)$ of maps described as follows.  First, the
factorization of $\Gd$ produces a map
$\RY:\Gr\times\Gr\to\Gr\times\Gr$ satisfying the set-theoretical
Yang-Baxter equation.  Second, $c$ is a natural isomorphism which
assigns to each pair of objects $(V,W)$ of $\cat$ an isomorphism
$V\otimes W\to W'\otimes V'$ where $(V',W')$ depends functorially on
$(V,W)$ and $c$ is a lift of $\RY$ to $\cat\times\cat$.  Given a
$\Gr$-ribbon category $\cat$ and a kind of section $\Gr\to Obj(\cat)$
the construction of \cite{KR05} produces a $\Gd$-link invariant.
As stated in the conclusion of \cite{KR05} the work of Kashaev and
Reshetikhin was motivated by the category of representation of the
unrestricted quantum groups at root of unity, in particular in the
case of quantum $\mathfrak gl_2$.  They show that even if this quantum
group does not have an $R$-matrix, the conjugation by the $R$-matrix
still induces a (not inner) automorphism of $U_q\mathfrak
gl_2^{\otimes2}$.  The point is that this automorphism induces a map
$\RY$ where the group $\Gr$ is obtained from the Poisson-Lie groups
factorization of $GL_2$ (see \cite{WX}).

The work of this paper is the beginning of a project to show that
$\Gd$-link invariant lead to invariants of 3-manifolds.  In this
context the motivating example comes from the representation theory
over a non-restricted quantum group at root of unity associated to any
simple Lie algebra $\g$, defined and studied by De Concini, Kac and
Procesi in \cite{DK, DKP, DKP2}.
The first step in obtaining such a 3-manifold invariant is to
construct a link invariant in the setting of these non-restricted
quantum groups.  In attempting to do this we ran into several
problems: First, the factorization $\Gd=\Gd_+\Gd_-$ is not strict in
this context as $\Gd_+\cap\Gd_-$ is not trivial.  Second, we were not
able to define a holonomy braiding for every pair of objects in the
category.  Third, we had problems checking the compatibility of the
partially defined holonomy braiding with the duality morphisms.
Finally, the quantum dimension of a generic $\Uxi(\g)$-module vanishes
which implies that the corresponding $\Gd$-link invariant is trivial.

In this paper we show that these problems can be overcome when $\Gd$
is the Borel of $SL_2(\C)$ and $\cat$ is the category of the so called
semi-cyclic representations of $\Uxi(\slt)$.  To do this, we first
slightly generalize the definition of a factorization of a group to
fit this example.  Then we use a section $\Gr\to\cat$ to show that it
is enough to consider a holonomy braiding defined on a generic set of
modules (this observation is already present in the work of
\cite{KR05}).  To avoid the third problem, we work with braids and a
$\Gd$-link generalization of Markov traces.  Finally, we re-normalize
the invariants to overcome the fourth problem.

The algebraic structures in the main example of this paper are related
to the work of Baseilhac in \cite{Bas2011} where he considers a bundle
over $\Gd$ of irreducible representation of $\Uxi(\slt)$ which is
relevant in 3-dimensional quantum hyperbolic field theories.

In \cite{ADO}, Akutsu, Deguchi, and Ohtsuki defined generalized
Alexander polynomial invariants of links.  These generalized
invariants can naturally be interpreted as maps on the space of
abelian representations of the fundamental group into $SL_2(\C)$, see
\cite{CGP}.  We call these invariants the \emph{nilpotent-ADO}
invariants.
The invariants of $\Gd$-links defined in this paper can be seen as the
natural extension of these maps to the space of reducible
representations of the fundamental group into $SL_2(\C)$ (a
representation $\rho$ is reducible if there is line in $\C^2$ which is
invariant under the action of the image of $\rho$).  Surprisingly,
experimental computations seem to indicate that this extension is
trivial.  In particular, the abelianization of a representation of the
fundamental group into the Borel of $SL_2(\C)$ is a cohomology class
in $H^1(S^3\setminus L,\C^*)$.  Then all of our computations suggest
that the invariants of this paper only depend on this abelianization
and so they seem to be equal to their corresponding nilpotent-ADO
invariants.

The paper is organized as follows. In Section \ref{S:GlinksandBraids}
we give the notions and some properties of factorized groups,
$\Gd$-links and $\Gr$-braids.  In Section \ref{GMarkovTraces} we
introduce $\Gr$-Markov trace and trace coloring systems.  In Section
\ref{S:ModTraces} we recall some basic results on modified traces,
which are the technical tools underlying our $\Gr$-Markov traces.
Section \ref{S:Example} is devoted to the example of the category of
semi-cyclic representations of $\Uxi(\slt)$.

\subsection{Acknowledgments} 
We would like to thank Nicolai Reshetikhin for many useful
discussions.  In particular, we thank him for his key suggestion to
use the semi-cyclic representations.  Also, Nicolai Reshetikhin and
Noah Snyder showed us a draft containing a factorization of the
$R$-matrix similar to that of $\check R_\xi$ in Theorem
\ref{T:3R-matrix}.

\section{Outline of how to define the invariant} 
In this section we give a short outline of how to use the general
results of this paper.  In Section~\ref{S:Example} we will apply these
results to the setting of $\Uxi(\slt)$.  The terms used in this
section will be defined rigorously in subsequent sections.

We start with a $\FK$-linear tensor category $\cat$ which is related
to a factorized group $(\Gd, \wb \Gd, \Gr)$ as follows.  Let
$\Yd\subset\Gd$ be a union of conjugacy classes in $\wb \Gd$ and $\Y$
be the corresponding set in $\Gr$.  Then there exists a bijection
$\RY:\Y\times\Y\to\Y\times\Y$ given by
$(x,y)\mapsto(\xl(x,y),\xr(x,y))$ which satisfies the Yang-Baxter
equation.  Here the functions $\xl$ and $\xr$ are defined using the
factorized group structure.  We require that the map $\RY$ is
compatible with a holonomy braiding: the category $\cat$ has a family
of objects $\A=\{A_y\}_{y\in \Y}$ and invertible morphisms
$$\{\B_{y,z}:A_y\otimes A_z\to A_{\xr(y,z)}\otimes A_{\xl(y,z)}\}_{y,z\in \Y}$$
satisfying the braid relations.  

A $\Yd$-admissible $\Gd$-link is a framed oriented link $L$ in $S^3$
equipped with a group homomorphism $\rho:\pi_1(S^3\setminus L,\oo)\to
\Gd$ such that $\rho(m)\in \Yd$ for any meridian $m$ of any component
of $L$.  If $\sigma$ is a braid with $n$ strands whose closure is $L$
then Proposition \ref{P:braid2} implies that $\rho$ induces a natural
$\Gr$-coloring on $\sigma$, i.e. a certain element $\y\in (\Y)^n$.
Then in the usual way the holonomy braiding and the braid generators
induces a map
$$
f_{(\sigma,\y)} : A_{y_1}\otimes A_{y_2}\otimes {\cdots} \otimes
A_{y_n}\to A_{y_1}\otimes A_{y_2}\otimes {\cdots} \otimes A_{y_n}.
$$

The final requirement on $\cat$ is the existence of a $\Gr$-Markov
trace which can be thought of as a family of twists $\{\theta_y\}_{
  y\in \Y}$ and a map $\t:\{f_{(\sigma,\y)}\}\to \FK$, where
$\{(\sigma,\y)\}$ are indexed by all $\Gr$-colored braids which
correspond to some $\Yd$-admissible $\Gd$-link.  The $\Gr$-Markov
trace satisfies some relations, including
\begin{equation}\label{E:OutlineMarkovTr}  
  \t\left(f_{(\sigma\sigma',\y)}\right)=\t\left(f_{(\sigma'\sigma,\y')}\right) 
  \;\;\; \text{ and }\;\;\;
  \t\left(f_{(\sigma_{n}^{\pm1}\sigma,\y'')}\right)= (\theta_y)^{\pm1}  
  \t\left(f_{(\sigma,\y)}\right) 
\end{equation}
where $\sigma_{n}$ is the $n$\textsuperscript{th} generator of braid
group on $n+1$ strands and $\y'$ and $\y''$ are certain tuples of
elements of $\Y$ determined by $(\sigma,\y)$.  If $(\sigma,y)$ is any
$\Gr$-colored braid corresponding to a $\Yd$-admissible $\Gd$-link
$L$, then Theorem~\ref{T:GlinkInv} states that
$F'(L)=\t\left(f_{(\sigma,\y)}\right)$ is a well defined invariant of
$L$.  The proof of the theorem is essentially that there exists a
finite number of $\Gr$-Markov moves relating any two $\Gr$-colored
braids $(\sigma,\y)$ and $(\sigma',\y')$ whose closures are isotopic
to the same $\Gd$-link.  Then the relations of
Equation~\eqref{E:OutlineMarkovTr} imply that the $\Gr$-Markov trace
preserves these $\Gr$-Markov moves and so leads to a well defined
invariant.

\section{${\Gd}$-links and $\Gr$-braids}\label{S:GlinksandBraids}

\newcommand{\mat}[2]{{\small
    \left(\begin{array}{cc}
      1&#2\\0&#1
    \end{array}\right)}}
\subsection{${\Gd}$-links and factorized groups}
Let $\Gd$ be a group.  A \emph{$G$-link} is a framed oriented link $L$
in $S^3$ and a group homomorphism $\rho:\pi_1(S^3\setminus L,\oo)\to
\Gd$.  We say $\rho$ is a representation on $L$.
If $\Yd$ is a union of conjugacy classes in $\wb\Gd$ of elements of
$\Gd$, we say that $\rho$ is a $\Yd$-admissible representation if
$\rho(m)\in \Yd$ for any meridian $m$ of any component of $L$.

Given a planar projection of $L$, the map $\rho$ can be extended to a
map on the meridian of the edges of the planar diagram.  This
extension is canonical in the context of a factorized group and leads
to a description of $G$-links as colored braids.  To do this we will
now define the notion of a \emph{factorized group} which is a tuple
$(\Gd, \wb \Gd, \Gr, \vp_+,\vp_-)$ where:

\begin{enumerate}
\item $\Gd,\wb \Gd$ and $\Gr$ are groups such that $\Gd$ is a normal
  subgroup of $\wb \Gd$,
\item there exists
group morphisms $\vp_+,\vp_-:\Gr\to\wb\Gd$ such that the map
$$
\psi:\Gr \to \wb\Gd \; \text{ given by } \; x \mapsto
\vp_+(x)\vp_-(x)^{-1}
$$
realize a bijection between $\Gr$ and $\Gd$.
\end{enumerate}
This notion of a factorized group is a slight generalization of the
treatment of factorizable groups given in \cite{KR05}.
 If
$x\in\Gd$, we denote $x_\pm=\vp_\pm(\psi^{-1}(x))\in\wb\Gd$.  So we
have $x=x_+x_-^{-1}$.  The map $\psi$ is not a group morphism but if
$x=\psi(x')$, $y=\psi(y')$ then $\psi(x'y')=x\,x_-yx_-^{-1}$ and
$\psi({x'}^{-1})=x_-^{-1}x^{-1}x_-$.  

\subsection{${\Gr}$-diagrams}\label{SS:diag}
Let $(\Gd, \wb \Gd, \Gr, \vp_+,\vp_-)$ be a factorized group.  Let $D$
be a regular planar projection of a framed oriented link $L$.  Then
$D$ is a quadrivalent graph embedded in $\R^2$ with the
``over-crossing information'' at its vertices.  Let $\ED$ be the set
of oriented edges of $D$.  A {\em $\Gr$-coloring of $D$} is a map
$\rho^*:\ED\to\Gr$ that satisfies
\begin{itemize}
\item for $e\in\ED$, we have $\rho^*(-e)=\rho^*(e)^{-1}$ where $-e$ is
  the edge $e$ with opposite orientation,
\item for any vertex $v$ of $D$, let $a,b,a',b'$ be the four edges
  adjacent to $v$ as shown in Figure~\ref{F:cross}, if $x=\rho^*(a),
  y=\rho^*(b), x'=\rho^*(a'), y'= \rho^*(b')\in \Gd$ then we have
\begin{figure}
  \centering
    $\epsh{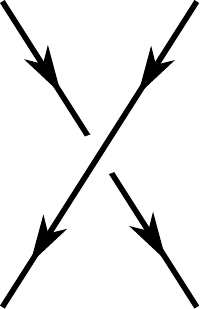}{12ex}\put(-33,-10){\mm{a}}\put(-5,-10){\mm{b}}
    \put(-34,11){\mm{b'}}\put(-5,11){\mm{a'}}$
 \caption{The edges at a crossing of $D$}\label{F:cross}
\end{figure}
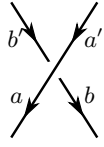  
  \begin{equation}
    \label{eq:p-cross}
    \left\{
      \begin{array}{ccc}
        y'\,x'&=&x\,y\in\Gr\\\vp_+(y')\vp_-(x')&=&\vp_-(x)\vp_+(y)\in\wb\Gd.
      \end{array}
    \right.
  \end{equation}
\end{itemize}

\begin{prop}[see also \cite{KR05}]\label{P:BijRepCol}
  There is a natural bijection between the set of $\Gd$-valued
  representations on $L$ and the set of $\Gr$-colorings of $D$:
  $$\qn{\rho:\pi_1(S^3\setminus L,\oo)\to \Gd} \simeq \qn{\rho^*:\ED\to \Gr}.$$ 
  Furthermore, if $\Yd\subset\Gd$ is a union of conjugacy classes in
  $\wb\Gd$ and $\Y=\psi^{-1}(\Yd)$ then this bijection identifies
  $\Yd$-admissible representations with $\Gr$-colorings with values in
  $\Y$.
\end{prop}
We will prove Proposition \ref{P:BijRepCol} in Section \ref{SS:bij}.
\subsection{${\Gr}$-braids}

Let $\HHH$ be a set and let $\RY: \HHH\times\HHH\to\HHH\times\HHH$ be
a bijective map that satisfies the set-theoretical Yang-Baxter
equation:
\begin{equation}
  \label{eq:YB}
  \RY_{23}\RY_{13}\RY_{12}=\RY_{12}\RY_{13}\RY_{23}
\end{equation}
here all mappings are from $(\HHH)^{\times 3}$ to itself and
$\RY_{ij}$ is the usual mapping (e.g.  the mapping $\RY_{12}$ acts as
$\RY$ in the first two factors and trivially in the last one).

Let $\tau:\HHH\times\HHH\to\HHH\times\HHH$ be the flip map defined by
$(y,y')\mapsto (y',y)$.  For $n\ge 1$, let
$\{\sigma_i\}_{i=1,{\ldots},n-1}$ be the usual generators of the braid
group $B_n$ on $n$ strands.  Let $\Bij(\HHH^n)$ be the group of
bijections of $\HHH^n$.  Consider the group homomorphism $B_n\to
\Bij(\HHH^n)$, $\sigma \mapsto \sigma_\RY$ induced by the assignment
$\sigma_i\mapsto ( \tau\circ\RY )_{i,i+1} $ where $ ( \tau\circ\RY
)_{i,i+1}$ is the bijection of $\HHH^n$ given by $( \tau\circ\RY )$ in
the $\text{i}^{\text{th}}$ and $\text{(i +1)}^{\text{th}}$ slots and
$\Id_\HHH$ otherwise.  A \emph{$\HHH$-coloring} of the braid $\sigma$
is a fixed point of this action, i.e.  an element
$\y=(y_1,\ldots,y_n)\in \HHH^n$ such that $\sigma_\RY(\y)=\y$.

We will now show that a factorized group gives rise to a natural
Yang-Baxter map $\RY$.
Let $(\Gd, \wb \Gd, \Gr, \vp_+,\vp_-)$ be a factorized group.  As
before, let $\Yd\subset\Gd$ be a union of conjugacy classes in $\wb\Gd$ and
$\Y=\psi^{-1}(\Yd)$. Let $\xr,\; \xl:\Gr \times \Gr\to \Gr $ be the
maps defined by
$$\xr(x,y)=\psi^{-1}\big(\vp_-(x)\psi(y)\vp_-(x)^{-1}\big)\; \text{ and } 
\;\xl(x,y)=\psi^{-1}\big(\vp_+(\xr(x,y))^{-1}\psi(x)\vp_+(\xr(x,y))\big).$$
\begin{prop}\label{P:braid}
  The map $\RY:\Gr\times\Gr\to \Gr\times\Gr$ defined by
  $(x,y)\mapsto(\xl(x,y),\xr(x,y))$ is bijective and satisfies the
  Yang-Baxter Equation \eqref{eq:YB}. Moreover, it restricts to a
  bijection $\RY:\Y\times\Y\to\Y\times\Y$.
\end{prop}
We will prove Proposition \ref{P:braid} in Section \ref{SS:bij}.  

\begin{prop}\label{P:braid2}  Let $\sigma\in B_n$ be a braid whose braid 
  closure is isotopic to a framed oriented link $L$.  There is a
  natural bijection between the set of $\Gd$-valued representations on
  $L$ and the set of $\Gr$-colorings of $\sigma$:
  $$\qn{\rho:\pi_1(S^3\setminus L,\oo)\to \Gd} \simeq 
  \qn{\y\in (\Gr)^n \; |\; \sigma_\RY(\y)=\y }.$$ Furthermore, this
  bijection identifies $\Yd$-admissible representation with
  $\Y$-coloring of $\sigma$.
\end{prop}
We will prove Proposition \ref{P:braid2} in Section \ref{SS:bij}.
\subsection{Example}\label{SS:ExG}
In this paper we will work with the following factorized group.  Let
$\wb\Gd$ be the group of $2\times2$ upper triangular invertible
matrices and $\Gd=\wb\Gd\cap\SL_2(\C)$. Let
$$\Gr=\qn{\mat\kappa\ve:\kappa\in\C^*,\ve\in\C}$$ and let 
$\vp_-:\Gr\to\wb\Gd$ be the inclusion map.   
Finally, let $\vp_+:\Gr\to\wb\Gd$ be the map given by
$$\vp_+\left(\mat\kappa\ve\right)={\small \left(\begin{array}{cc}
      \kappa&0\\0&1
    \end{array}\right)}.$$
Thus we have factorized the upper triangular matrices of $\SL_2(\C)$
as
$$ {\small \left(\begin{array}{cc}
      \kappa&\ve\\0&\kappa^{-1}
    \end{array}\right)}=\psi\bp{\small \left(\begin{array}{cc}
      1&-\ve\\0&\kappa
    \end{array}\right)}={\small \left(\begin{array}{cc}
      \kappa&0\\0&1
    \end{array}\right)}\times\mat\kappa{-\ve}^{-1}.$$
Then Proposition \ref{P:braid} gives a Yang-Baxter map:
$$\RY\bp{\mat{\kappa_1}{\ve_1},\mat{\kappa_2}{\ve_2}}=
\bp{\mat{\kappa_1}{\ve_1/\kappa_2},
\mat{\kappa_2}{(\ve_2+\ve_1(\kappa_2-\kappa_2^{-1}))/\kappa_1}}.$$ 
In this example, we define
$\Y=\qn{\mat\kappa\ve:\kappa\in\C^*\setminus\{\pm1\},\ve\in\C}$.  
Then $\Yd=\psi(\Y)$ is the set of matrices in $\Gd$ whose trace is not
$\pm2$ and so $\Yd$ is a union of conjugacy classes in $\wb \Gd$. 
\subsection{Proofs of propositions}\label{SS:bij}
In this subsection we will prove Propositions \ref{P:BijRepCol},
\ref{P:braid} and \ref{P:braid2} using the ideas of \cite{KR05}.  To
do this we first need to develop some terminology.

Let $L$ be a framed oriented link $L$ in $S^3=\R^3\cup \oo$.  Let
$D\subset\R^2\times\{0\}$ be a regular planar projection of $L$.
The diagram $D$ splits $\R^2\times\{0\}$ into regions
$R_\oo,R_1,\ldots,R_n$ (where we assume that $R_\oo$ is the infinite
region).  For each choose a point $M_i\times \{0\}$ in $R_i\subset
\R^2\times\{0\}$.  Then the fundamental group $\pi_1(S^3\setminus
L,\oo)$ is generated by the downward oriented vertical lines $\wv
R_i=M_i\times \R$, which are loops through $\oo$.  Here $\wv R_\oo$ is
the trivial loop.  Split the loop $\wv R_i$ into two paths $\wv
R_i^-=M_i\times \R^-$ and $\wv R_i^+=M_i\times \R^+$ both oriented
from $\oo$ to $M_i$ (see Figure \ref{F:link}(A)).
\begin{figure}
 \begin{minipage}[b]{.5\linewidth}
  \centering
    $\epsh{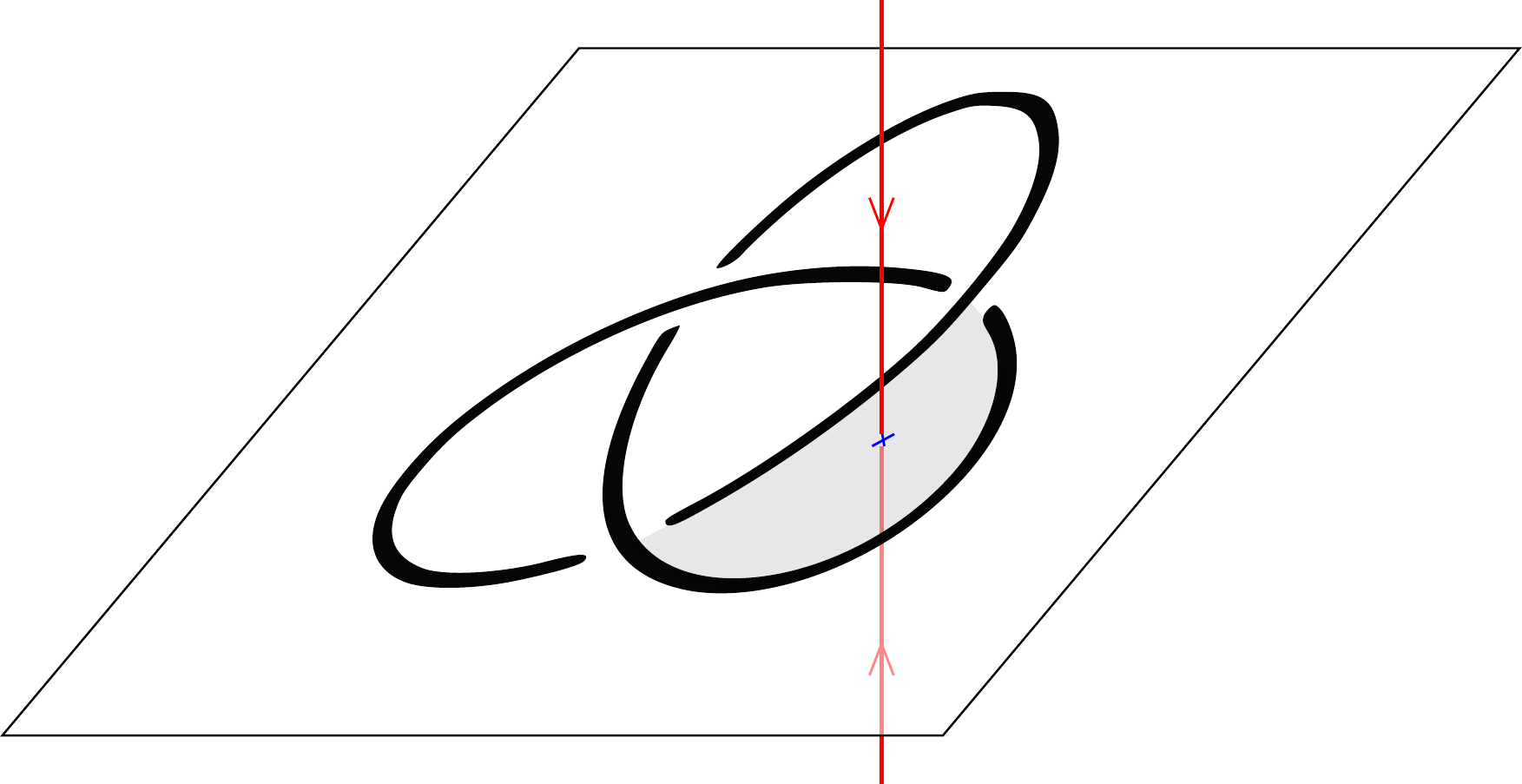}{28ex}\put(-94,-5){\ms{\color{blue}M_i}}
    \put(-94,30){\mm{\color{red}\wv R_i^+}}
    \put(-94,-40){\mm{\color{red}\wv R_i^-}}$
    \\[1ex]
   (A) The diagram of a link, a region $R_i$\\ and the associated
   path in $\wb\pi$
 \end{minipage} %\hfill
 \begin{minipage}[b]{.4\linewidth}
  \centering
 $\epsh{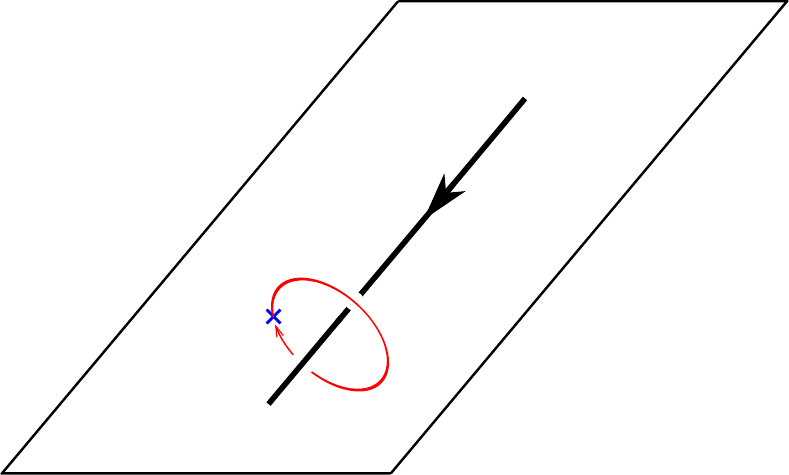}{28ex}\put(-143,-15){\ms{\color{blue}M_i}}
    \put(-99,-40){\mm{\color{red}m_e }}
    \put(-90,0){\mm{e}}$ 
    \\[1ex]
   (B) An oriented edge $e$ of the diagram and its meridian $m_e$.
 \end{minipage}
 \caption{Diagrams of a link}\label{F:link}
\end{figure}
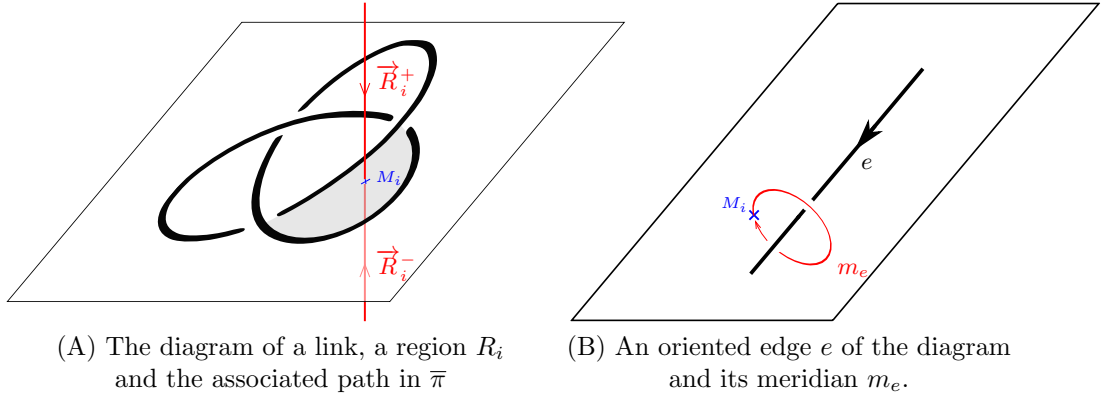
In the groupoid $\wb\pi=\pi_1(S^3\setminus L,\{M_i\})$, we have $\wv
R_i=\wv R_i^+\wov R_i^-$ (where $\wov X$ is defined to be the path
$\wv X$ in the opposite direction and we use the contravariant
concatenation of paths in which $\gamma\delta$ is defined when
$\gamma(1)=\delta(0)$).

As above, let $\ED$ be the set of oriented edges of $D$.  If $e\in\ED$
is an oriented edge with regions $R_i$ on the right and $R_j$ on the
left, we define the meridian of $e$ as the loop based at $M_i$ given
by $m_e=\wov R_i^+\wv R_j^+\wov R_j^-\wv R_i^-\in\wb\pi$, see Figure
\ref{F:link}(B).  (Note the orientation of $m_e$ is opposite to the
one induced by $e$.)  We have $m_e=m_e^+(m_e^-)^{-1}$ where
$m_e^+=\wov R_i^+\wv R_j^+$ and $m_e^-=\wov R_i^-\wv R_j^-$.  In $
\wb\pi$, also we have $m_e=\wov R_i^+\wv R_j\wov R_i\wv R_i^+$ which
is conjugate to the loop $\wv R_j\wov R_i$ based at $\oo$.

The following lemma can be deduced from the classical presentation of
$\pi_1(S^3\setminus L,\oo)$ associated to the diagram $D$.
\begin{lem}\label{L:GroupoidGen}
  The groupoid $\wb \pi=\pi_1(S^3\setminus L,\{M_i\})$ is generated by
  the elements $\{m^+_e,m^-_e:e\in \ED\}$ subject to the following
  relations~: 
\begin{itemize}
\item For any $e\in\ED$, if $-e$ is the same edge with opposite
  orientation, then $m_{-e}^{\pm}=(m_{e}^{\pm})^{-1}$.
\item For any vertex $v$ of $D$, if $a,b,a',b'$ are the four edges
  adjacent to $v$ as shown in Figure \ref{F:cross} then one has
  \begin{equation}
    \label{eq:m-cross}
    \left\{
      \begin{array}{ccc}
        m_{b'}^+\,m_{a'}^+&=&m_{a}^+\,m_{b}^+\\
        m_{b'}^-\,m_{a'}^-&=&m_{a}^-\,m_{b}^-\\
        m_{b'}^+\,m_{a'}^-&=&m_{a}^-\,m_{b}^+.
      \end{array}
    \right.
  \end{equation}
\end{itemize}  
\end{lem}

\begin{proof}[Proof of Proposition \ref{P:BijRepCol}]
  Let $\rho:\pi_1(S^3\setminus L,\oo)\to G$ be a representation on
  $L$.  Then $\rho$ can be extended to a map of groupoids
  $\wb\rho:\wb\pi\to\wb\Gd$ by setting $\wb\rho(\wv R_i^\pm)=\rho(\wv
  R_i)_\pm$.  In particular, we have $\wb\rho(\wv
  R_i)=\rho(R_i)\in\Gd$ and $\wb\rho(m_e)=\wb\rho(\wv
  R_i^+)^{-1}\rho(\wv R_j\wov R_i)\wb\rho(\wv R_i^+)\in\wb\rho(\wv
  R_i^+)^{-1}\,\Gd\,\wb\rho(\wv R_i^+)=\Gd$.  We define a map
  $\rho^*:\ED\to\Gr$ that expresses %so that $\psi\circ \rho^*$ is
  the monodromy around edges: for $e\in\ED$ with meridian $m_e$, let
  $\rho^*(e)=\psi^{-1}(\wb\rho(m_e))$.  Now Lemma \ref{L:GroupoidGen}
  implies
  $$\psi(\rho^*(-e))=\wb\rho(m_{-e})
  =\wb\rho(m_{-e}^{+})\wb\rho(m_{-e}^{-})^{-1}
  =\wb\rho(m_{e}^{+})^{-1}\wb\rho(m_{e}^{-})=\psi(\rho^*(e)^{-1}).$$
  Moreover, at a vertex $v$ as shown in Figure \ref{F:cross}, the two
  first lines of Equation \eqref{eq:m-cross} imply
  $\vp_+(\rho^*(b'a'))=\vp_+(\rho^*(ab))$ and
  $\vp_-(\rho^*(b'a'))=\vp_-(\rho^*(ab))$.  But clearly the map
  $\vp_+\times\vp_-:\Gr\to\wb\Gd\times\wb\Gd$ is injective as $\psi$
  -- which is bijective -- factors through this map.  Therefore, we
  have $\rho^*(b'a')=\rho^*(ab)$.  The last line of Equation
  \eqref{eq:m-cross} shows
  $\vp_+(\rho^*(b'))\vp_-(\rho^*(a'))=\vp_-(\rho^*(a))\vp_+(\rho^*(b))$.
  Thus, $\rho^*$ is a $\Gr$-coloring of $D$.

  Furthermore, if $\rho$ is a $\Yd$-admissible representation then
  since $\Yd$ is a union of conjugacy classes in $\wb\Gd$ we have
  $\wb\rho(m_e)=\wb\rho(\wv R_i^+)^{-1}\rho(\wv R_j\wov
  R_i)\wb\rho(\wv R_i^+)\in\wb\rho(\wv R_i^+)^{-1}\,\Yd\,\wb\rho(\wv
  R_i^+)\subset \Yd$.  Since $\Y=\psi^{-1}(\Yd)$ we have $\rho^*$ is a
  $\Gr$-coloring with values in $\Y$.

  Conversely, suppose that $\rho^*$ is a $\Gr$-coloring of $D$.  The
  definition of a $\Gr$-coloring implies that the map
  $\wb\rho:\wb\pi\to\wb\Gd$ defined by
  $\wb\rho(m_e^{\pm})=\vp_\pm(\rho^*(e))$ satisfy the relations of
  Equation \eqref{eq:m-cross} and thus defines a well defined groupoid
  morphism.  The morphism $\wb \rho$ is $\Gd$-valued on meridians and
  their products.  Therefore, if $R_i$ and $R_j$ are adjacent regions,
  then the loop $\wv R_j\wov R_i$ is conjugated to a meridian and
  since $\Gd$ is a normal subgroup of $\wb \Gd$ we have $\wb\rho(\wv
  R_j\wov R_i)\in\Gd$.  As these loops generate $\pi_1(S^3\setminus
  L,\oo)$, we get that the restriction of $\wb\rho$ to
  $\pi_1(S^3\setminus L,\oo)$ takes values in $\Gd$.  Finally, if
  $\rho^*$ is a $\Gr$-coloring with values in $\Y$, then the
  restriction of $\wb\rho$ to $\pi_1(S^3\setminus L,\oo)$ takes values
  in $\Yd$, since $\Yd$ is a union of conjugacy classes in $\wb\Gd$.
\renewcommand{\qedsymbol}{\fbox{\ref{P:BijRepCol}}}
\end{proof}\renewcommand{\qedsymbol}{\fbox{\thetheo}}

\begin{proof}[Proof of Proposition \ref{P:braid}]
  Let $x,y\in\Gr$.  Let $(x',y')\in\Gr$ be the unique tuple which is a
  solution to the system of equations
  \begin{equation}\label{eq:p-cross2} y'\,x'=x\,y \text{ and }
    \vp_+(y')\vp_-(x')=\vp_-(x)\vp_+(y)
  \end{equation}
  (note these equations are exactly part of the requirements of a
  $\Gr$-coloring, see Equations \eqref{eq:p-cross}). These equations
  are equivalent to the equation: $\RY(x,y)=(x',y')$. So the map $\RY$
  is determined by Equations \eqref{eq:p-cross2}.  Using this, for
  $x,y,z\in\Gr$, we have
  $$(x',y',z')=\RY_{23}\RY_{13}\RY_{12}(x,y,z)  \iff
  \left\{\begin{array}{l}
      z'y'x'=xyz\\
      \vp_+(z')\vp_-(y'x')=\vp_-(xy)\vp_+(z)\\
      \vp_+(z'y')\vp_-(x')=\vp_-(x)\vp_+(yz)
    \end{array}
  \right.
  $$
  $$\iff(x',y',z')=\RY_{12}\RY_{13}\RY_{23}(x,y,z)$$

  \begin{comment}
  $$(x',y',z')=\RY_{23}\RY_{13}\RY_{12}(x,y,z)$$
  $$
  \iff
  \left\{\begin{array}{l}
      (x_1,y_1)=\RY(x,y)\\(x',z_1)=\RY(x_1,z)\\(y',z')=\RY(y_1,z_1)
    \end{array}
  \right.
  $$
  % 
  $$\iff
  \left\{\begin{array}{l}
      y_1x_1=xy,\, \vp_+(y_1)\vp_-(x_1)=\vp_-(x)\vp_+(y)\\
      z_1x'=x_1z,\, \vp_+(z_1)\vp_-(x')=\vp_-(x_1)\vp_+(z)\\
      z'y'=y_1z_1,\, \vp_+(z')\vp_-(y')=\vp_-(y_1)\vp_+(z_1)
    \end{array}
  \right.
  $$
  % 
  $$\iff
  \left\{\begin{array}{l}
      z'y'x'=y_1z_1x'=y_1x_1z=xyz\\
      \vp_+(z')\vp_-(y')\vp_-(x')=\vp_-(y_1)\vp_+(z_1)\vp_-(x')=
      \vp_-(y_1)\vp_-(x_1)\vp_+(z)
      =\vp_-(x)\vp_-(y)\vp_+(z)\\
      \vp_+(z')\vp_+(y')\vp_-(x')=\vp_+(y_1)\vp_+(z_1)\vp_-(x')=
      \vp_+(y_1)\vp_-(x_1)\vp_+(z)
      =\vp_-(x)\vp_+(y)\vp_+(z)
    \end{array}
  \right.
  $$
  \end{comment}
  % 
Finally, since $\psi(\xl(x,y))$ is conjugated to $\psi(x)$ and
  $\psi(\xr(x,y))$ is conjugated to $\psi(y)$ we have that
  $\RY(\Y\times\Y)= \Y\times\Y$. 
\renewcommand{\qedsymbol}{\fbox{\ref{P:braid}}}
\end{proof}\renewcommand{\qedsymbol}{\fbox{\thetheo}}

\begin{proof}[Proof of Proposition \ref{P:braid2}]
  Let $\sigma\in B_n$ be a braid whose braid closure is isotopic to
  $L$.  Represent the closure of $\sigma$ by a planar diagram $D$
  encoded by the braid word corresponding to $\sigma$ in the
  generators $\{\sigma_i\}_{i=1\cdots n-1}$ (see Figure \ref{F:BD}).
\begin{figure}
 \begin{minipage}[b]{.2\linewidth}
  \centering
    $\epsh{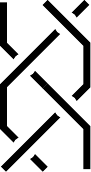}{12ex}$
 \end{minipage} %\hfill
 \begin{minipage}[b]{.4\linewidth}
   \centering $\epsh{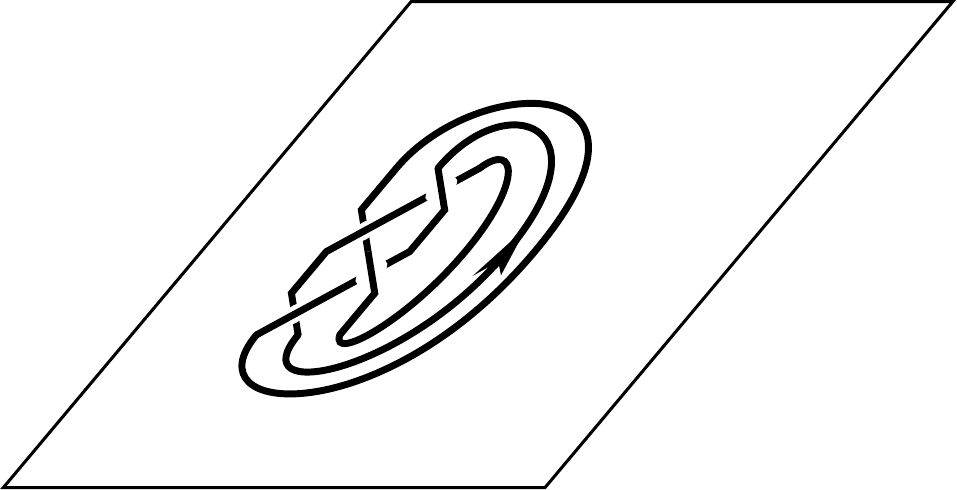}{18ex}$
 \end{minipage}
 \caption{Diagrams of the braid word
   $\sigma_2^{-1}\sigma_1\sigma_2^{-1}\sigma_1$ and of its braid
   closure which is a regular projection of the figure eight
   knot.}\label{F:BD}
\end{figure}
Proposition \ref{P:BijRepCol} gives a one-to-one correspondence
between $\Gd$-valued representations on $L$ and $\Gr$-colorings of
$D$.  So to prove the proposition it suffices to show that the
$\Gr$-colorings of $D$ are in one-to-one correspondence with
$\Gr$-coloring of $\sigma$.

Let $\y=(y_1,\ldots,y_n)$ be a $\Gr$-coloring of $\sigma$,
i.e. $\sigma_\RY(\y)=\y$.  Then we color the edges of $D$ at the
bottom of the diagram by $y_1,\ldots,y_n$ and use $\RY$ and its
inverse to propagate these colors to the top of the diagram.  The
result is a $\Gr$-coloring of $D$.  Conversely, for any $\Gr$-coloring
of the diagram $D$, one obtains a $\Gr$-coloring of $\sigma$.  These
are clearly inverse correspondences.
 
\renewcommand{\qedsymbol}{\fbox{\ref{P:braid2}}}
\end{proof}\renewcommand{\qedsymbol}{\fbox{\thetheo}}

\section{$\Gr$-Markov traces}\label{GMarkovTraces}
In this section we give the notion of $\Gr$-Markov moves and coloring systems.  
We fix in this section a factorized group $(\Gd, \wb \Gd, \Gr,
\vp_+,\vp_-)$, $\Yd\subset\Gd$ a union of conjugacy classes in $\wb\Gd$ and
$\Y=\psi^{-1}(\Yd)$.
Recall the following well known theorem  of Markov.
\begin{theo}[Markov] \label{T:Markov}
Any unframed oriented link
in $S^3$ can be represented as the closure of a braid.  Furthermore,
any two braids whose closures are ambient isotopic links are related
by a finite sequence of the following Markov moves:
  $$
  \begin{array}{ccc}
    \sigma\sigma'&\longleftrightarrow& \sigma'\sigma\\
    \sigma&\longleftrightarrow&\sigma_{n}^{\pm1}\sigma
  \end{array}
  $$
  where $n\in\N^*$, $\sigma,\sigma'\in B_n$ and $\sigma_n$ is the
  $n$\textsuperscript{th} generator of $B_{n+1}$.
\end{theo}

The following proposition is a $\Gd$-link version of the previous
theorem.
\begin{prop}\label{P:G-Markov}
  Let $\shift:\Gr \to \Gr$ be the map defined by
  $\shift(x)=\psi^{-1}\big(\vp_-(x)^{-1}\psi(x)\vp_-(x)\big)$.  Any
  unframed oriented $\Gd$-link
  in $S^3$ can be represented as the closure of a $\Gr$-colored braid.
  Furthermore, any two $\Gr$-colored braids whose closures are ambient
  isotopic $\Gd$-links are related by a finite sequence of the
  following $\Gr$-Markov moves:
  $$
  \begin{array}{ccc}
    (\sigma\sigma',\y)&\longleftrightarrow&(\sigma'\sigma,\sigma_\RY(\y))\\
    (\sigma,\y)&\longleftrightarrow&
    (\sigma_{n}^{\pm1}\sigma,(y_1,\ldots,y_n,\shift^{\pm1}(y_n)))
  \end{array}
  $$
  where $\y=(y_1,\ldots,y_n)$.
\end{prop}
\begin{proof}
  Let $L$ be a $\Gd$-link.  Then by Theorem \ref{T:Markov} the link
  underlying $L$ can be represented as the closure of a braid
  $\sigma$.  Then Proposition \ref{P:braid2} implies that the
  $\Gd$-valued representation on $L$ gives a $\Gr$-coloring on
  $\sigma$.  This proves the first part of the proposition.

  To prove the second part of the proposition we will first prove that
  the $\Gr$-Markov moves preserve the isotopy class of the closure of
  the corresponding braids.  For the first move, take a diagram $D$
  obtained as the closure of the concatenation of a diagram
  representing $\sigma$ and a diagram representing $\sigma'$.  The
  representation on $L$ induces a unique $\Gr$-coloring on $D$, which
  can be used to deduce unique $\Gr$-coloring on both $\sigma\sigma'$
  and $\sigma'\sigma$.  Clearly if the colors of the strands at the
  bottom of $\sigma$ are given by $\y$ then the colors of the strands
  at the top of $\sigma$ (i.e. the bottom of $\sigma'$) are given by
  $\sigma_\RY(\y)$.

  For the second move, take a diagram $D$ obtained as the closure of a
  diagram representing $\sigma$ and consider its $\Gr$-coloring. Then
  a diagram $D'$ for $\sigma'=\sigma_{n}\sigma$ is obtained by doing a
  Reidemeister I move to $D$.  The $\Gr$-coloring of $D'$ is the
  coloring on $D$ except with an additional color $x$ on the loop
  coming from the Reidemeister move.  Now Equation \eqref{eq:p-cross}
  at the new crossing of $D'$ implies that,
  $$(x,y_n)=\RY(y_n,x)\iff \vp_+(y_n)\vp_-(x)=\vp_-(y_n)\vp_+(x)
  \iff$$ $$\psi(x)=\vp_-(y_n)^{-1}\vp_+(y_n)\iff x=\shift(y_n).$$
  Similarly if $\sigma'=\sigma_{n}^{-1}\sigma$, we have
  $(x,y_n)=\RY^{-1}(y_n,x)
  \iff x=\shift^{-1}(y_n)$.  Thus, the two $\Gr$-colorings of the two
  braids in the second move are related as in the proposition.
  Moreover, from what we just proved $\shift^{\pm}(y_n)$ is the unique
  value of $\Gr$ which satisfies the second move.
  
  Finally, let $\sigma$ and $\sigma'$ be two $\Gr$-colored braids
  whose closures are ambient isotopic to the $\Gd$-link $L$.  By
  Theorem \ref{T:Markov} the braids $\sigma$ and $\sigma'$ are related
  by a finite sequence of Markov moves.  Then from what we just proved
  the $\Gr$-coloring is uniquely propagated through these moves by the
  corresponding $\Gr$-Markov move.  Thus, $\sigma$ and $\sigma'$ are
  related by a finite sequence of $\Gr$-Markov moves.
\end{proof} 

Next we give the notion of a coloring system associated to the pair
$(\Y,\RY)$.  Let $\cat$ be a $\FK$-linear tensor category.  Let $I$ be
a set.  Let $\{A^i_y\}_{i\in I, y\in \Y}$ be a family of objects in
$\cat$.  Let $\A=\{A^i_y:\,i\in I,y\in\Y\}$ and $\A^{\otimes
  n}=\{V_1\otimes\cdots \otimes V_n:\,(V_1,\ldots, V_n)\in\A^n\}$.

A family of invertible morphisms 
$$\B_{y,z}^{i,j}:A^i_y\otimes A^j_z\to A^j_{\xr(y,z)}\otimes A^i_{\xl(y,z)}$$
in $\cat$ indexed by $y,z\in\Y$ and $i,j\in I$ form a \emph{holonomy
  braiding} if for all $ x,y,z\in\Y$ and $ i,j,k\in I$ the following
equality holds
 $$
 \bp{\B_{y',z'}^{j,k}\otimes\Id}\bp{\Id\otimes\B_{x',z}^{i,k}}
 \bp{\B_{x,y}^{i,j}\otimes\Id}
 =\bp{\Id\otimes\B_{x'',y''}^{i,j}}\bp{\B_{x,z''}^{i,k}\otimes\Id}
 \bp{\Id\otimes\B_{y,z}^{j,k}}
 $$
 where $ x'=\xl(x,y), y'=\xr(x,y), z'=\xr(x',z), x''=\xl(x,z''),
 y''=\xl(y,z)$ and $z''=\xr(y,z)$.  A family of linear maps
 $\tt=\{\tt_V: \End_\cat(V)\to\FK\} $ indexed by $V\in\A^{\otimes n}$
 and $ n\in \N^*= \{1,2,3,{\ldots}\}$ is a \emph{$\Gr$-Markov trace}
 if for all $n\in\N^*$ and $V,W\in \A^{\otimes n}$ the following hold
  \begin{enumerate}
  \item  $\tt_W(f g)=\tt_V(g f)$ for any $f:V\to W$, $g:W\to V$,
  \item there exists a set of \emph{twists} $\{\theta^i_y\}_{i\in I,
      y\in \Y}$ of invertible elements of $\FK$ such that
    \begin{enumerate}
    \item for $y,y'\in\Y$, if $\psi(y)$ and $\psi(y')$ are conjugate
      in $\wb\Gd$ then $ \theta^i_y=\theta^i_{y'}$ for all $i\in I$,
    \item $\tt_{V\otimes A_y^i\otimes A_{\shift(y)}^i}
      \big((\Id\otimes(\B_{y,\shift(y)}^{i,i})^{\pm1})(h\otimes\Id)\big)
      =(\theta^i_y)^{\pm1}\tt_{V\otimes A_y^i}(h)$ for any
      $h\in\End_\cat(V\otimes A_y^i)$.
    \end{enumerate}  
  \end{enumerate}
  A tuple $(\Y,\RY,\{A^i_y\},\{\B_{y,z}^{i,j}\},\tt)$ is a \emph{trace
    coloring system} if $\{\B_{y,z}^{i,j}\}$ is a holonomy braiding
  and $\tt$ is a $\Gr$-Markov trace.  Given such a tuple we can assign
  an endomorphism to a $\Y$-colored braid representing a
  $\Yd$-admissible $\Gd$-link whose components are colored by elements
  of $I$.  Consider the pair $(\sigma_k,\y)$ where $\sigma_k$ is a
  generator of $B_n$ whose $l^{\text{th}}$ strand is colored by
  $i_l\in I$ and $\y=(y_1,{\ldots},y_n)\in \Y^n$.  Then
  $(\sigma_k,\y)$ induces the homomorphism
  \begin{equation}\label{E:IdCId}
    \Id_{A_{y_1}^{i_1}}\otimes {\cdots}\otimes \Id_{A_{y_{k-1}}^{i_{k-1}}} 
    \otimes \B_{y_k,y_{k+1}}^{i_k,i_{k+1}} \otimes \Id_{A_{y_{k+2}}^{i_{k+2}}} 
    \otimes{\cdots}\otimes  \Id_{A_{y_{n}}^{i_{n}}}.
\end{equation}
Let $B_n^{\Y,I}$ be the set of $\Y$-braids $(\sigma,\y)$ with $n$
strands such that the components of the closure of $\sigma$ are
colored by elements of $I$.  Let $(\sigma,\y)\in B_n^{\Y,I}$ then the
$I$-coloring induces a map from the strands of $\sigma$ to $I$ (say
the $l^{\text{th}}$ strand is colored by $i_l\in I$).  Thus, Equation
\eqref{E:IdCId} induces an endomorphism
$$ 
f_{(\sigma,\y)} : A_{y_1}^{i_1}\otimes A_{y_2}^{i_2}\otimes {\cdots}
\otimes A_{y_n}^{i_n}\to A_{y_1}^{i_1}\otimes A_{y_2}^{i_2}\otimes
{\cdots} \otimes A_{y_n}^{i_n}
$$
where $\y=(y_1,{\ldots},y_n)$.  We write
$\tt(f_{(\sigma,\y)})=\tt_{A_{y_1}^{i_1}\otimes {\cdots} \otimes
  A_{y_n}^{i_n}}(f_{(\sigma,\y)})$.

\begin{theo}\label{T:GlinkInv}  
  Let $(\Y,\RY,\{A^i_y\},\{\B_{y,z}^{i,j}\},\tt)$ be a trace coloring
  system in a $\FK$-linear tensor category.  Let $(L,\rho)$ be a
  $\Yd$-admissible $\Gd$-link whose components are colored with
  elements of $I$.  Let $(\sigma,\y)$ be a $\Y$-colored braid
  representing $L$.  Then
  $$F'(L,\rho)=\tt(f_{(\sigma,\y)})$$
  is independent of the choice of $(\sigma,\y)$ and yields a well
  defined invariant of the $\Gd$-link $(L,\rho)$.
\end{theo}

\section{Modified right traces on right ideals}\label{S:ModTraces}
There are many examples where the usual categorical trace is
generically zero and invariant of Theorem \ref{T:GlinkInv} is trivial
with this trace.  Here we recall the modified trace construction given
in \cite{GPV}.  These traces can be non-zero when the usual
categorical trace is zero.  Later in the paper we will show that the
modified traces lead to $\Gr$-Markov traces and non-trivial link
invariants.

In this section we recall some properties about traces on ideals, for
more details see \cite{GPV}.  Let $\kk$ be a domain.  A \emph{monoidal
  \kt category} is a strict monoidal category $\cat$ such that its
hom-sets are \kt modules, the composition and monoidal product of
morphisms are \kt bilinear, and $\End_\cat(\unit)$ is a free \kt
module of rank one, where $\unit$ is the unit object.  An object $X$
of a monoidal \kt category $\cat$ is \emph{simple} if $\End_\cat(X)$
is a free \kt module of rank 1.  Equivalently, $X$ is simple if the
\kt homomorphism $\kk \to
\End_\cat(X),\, k   \mapsto  k\, \Id_X$  is an isomorphism. 

Let $\cat$ be a pivotal \kt category (see for example \cite{BW,
  Malt}), with unit object~$\unit$, duality morphisms
\begin{align*}
  & \ev_X \colon X^*\otimes X \to\unit,
  \quad \coev_X\colon \unit  \to X \otimes X^*,\\
  & \tev_X \colon X\otimes X^* \to\unit, \quad \tcoev_X\colon \unit
  \to X^* \otimes X.
\end{align*}
Recall that in $\cat$, the left dual and right dual of a morphism
$f\colon X \to Y$ in $\cat$ coincide:
\begin{align*}
  f^*&= (\ev_Y \otimes \Id_{X^*})(\Id_{Y^*} \otimes f \otimes
  \Id_{X^*}) (\Id_{Y^*}\otimes \coev_X)
  \\
  &= (\Id_{X^*} \otimes \tev_Y)(\Id_{X^*} \otimes f \otimes
  \Id_{Y^*})(\tcoev_X \otimes \Id_{Y^*}) \colon Y^*\to X^*.
\end{align*}

For $X,Y,Z\in\cat$, the \emph{right partial trace} (with respect to
$X$) is the map $\tr_r^X \colon \Hom_\cat(Y \otimes X, Z \otimes X)
\to \Hom_\cat(Y,Z)$ defined, for $g \in \Hom_\cat(Y \otimes X, Z
\otimes X)$ by
$$
\tr_r^X(g)=(\Id_Z \otimes \tev_X)(g \otimes \Id_{X^*})(\Id_Y \otimes \coev_X).
$$

By a \emph{retract} of an object $X$ of a category $\cat$, we mean an
object $U$ of~$\cat$ such that there exists morphisms $p\colon X \to
U$ and $q\colon U \to X$ verifying $pq=\Id_U$.  By a \emph{right}
\emph{ideal} of a monoidal category $\cat$, we mean a class $\ideal
\subset \cat$ such that the following two conditions hold:
\begin{enumerate}
\item If $X\in \ideal$ and $Y\in\cat$ then $X\otimes Y\in \ideal$.
\item Any retract (in $\cat$) of an object of $\ideal$ belongs to $\ideal$.
\end{enumerate}
Given an object $X$ of $\cat$ we can define the right ideal  $ \ideal^r_X$ as follows
\begin{align*}
& \ideal^r_X=\bigl\{ U \in \cat \, \big | \, \text{$U$ is a retract of $X \otimes Z$ for some  $Z\in\cat$}  \bigr \}.
\end{align*}

 A \emph{right trace} on a right ideal  $\ideal$ of $\cat$ is
a family $\t=\{\t_X\colon \End_\cat(X)\rightarrow \kk\}_{X\in\ideal}$
of \kt linear forms
such that
\begin{equation*}\label{defrtrace}
\t_{X\otimes Z}(f)=\t_X(\tr_r^Z(f)) \quad \text{and} \quad \t_V(gh)=\t_U(hg)
\end{equation*}
for any $f\in \End_\cat(X\otimes Z)$,  $g \in \Hom_\cat(U,V)$, and $h \in \Hom_\cat(V,U)$, with $X,U,V\in \ideal$ and $Z\in \cat$.

Since $\cat$ is pivotal, then $\phi=\{\phi_X=(\tev_{X}\otimes\Id_{X^{**}})(\Id_X\otimes\coev_{X^{*}})\colon X\to
X^{**}\}_{X \in \cat}
$
is a monoidal natural isomorphism.   Let $X$ be an object of $\cat$ and  
$t \colon \End_\cat(X)\rightarrow \kk$
a \kt linear map.  We say that $t$ is a  \emph{right ambidextrous trace} on $X$ if
\begin{equation*}
t\bigl(\phi_X^{-1} \bigl(\tr_r^{X}\!(f)\bigr)^* \phi_X\bigr)=t\bigl(\tr_l^{X^*}\!(f)\bigr)
\end{equation*}
for all $g \in \End_\cat(X^* \otimes X)$.

 If $X$ is a simple object of $\cat$, we
denote by $\brk{\,}_X \colon \End_\cat(X) \to \kk$ the inverse of the \kt
linear isomorphism $\kk \to \End_\cat(X)$ defined by $k \mapsto k\, \Id_X$.
We say that an object $X$ of $\cat$ is \emph{right ambi} if it is simple and
the \kt linear form $\brk{\,}_X$ is a right
ambidextrous trace on $X$.

\begin{theo}[\cite{GPV}]\label{T:traceXtraceI}
 If $t$ is a right ambidextrous trace on an object $X$ of $\cat$, then there exists a unique right trace
  $\t$ on $\ideal_X^r$ such that
  $\t_{X}=t$.
\end{theo}
\begin{proof}
The theorem is a special case of Theorem 10 in \cite{GPV} where ${\mathcal O}=\{X\}$.
\end{proof}

\section{Example~: Semi-cyclic modules of $\Uxi$}\label{S:Example}

In this section we prove that the so called semi-cyclic modules of
$\Uxi$ can be used to define a trace coloring system.  We do this in
six steps (contained in six subsections): 1) we define several
algebras, 2) consider a completion, 3) use this completion to define a
$R$-matrix, 4) show that this $R$-matrix acts on certain modules
giving a holonomy braiding, 5) define a right trace and 6) combine the
results of the section to define a trace coloring system.  7) is an
explicit computation of the holonomy $R$-matrix at fourth root of
unity.

In this section we consider seven versions of quantized $\slt$:
$$
\begin{array}{ccccccccc}
  &&&&\UqHe&\subset&\wh\UqHe&\subset& U_h(\slt)\\
  &&&&\downarrow&&\downarrow\\
  \Uxi&\subset& \UxiH&\subset&\UxiHe&\subset&\Uw
\end{array}
$$
The character ``$h$'' indicates $U_h(\slt)$ is a $\C[[h]]$-topological
algebra, the character ``$q$'' indicates that the corresponding
algebra is a $\Z[q^{1}, q^{-1}]$ algebra and the character ``$\xi$''
indicates a specialization at a root of unity $\xi\in\C$.  The
character ``$D$'' is for divided power (of the generator $F$) and the
hat suggests a completion.

\subsection{An integral version of $U_q\slt$}

We use the following notation.  
\begin{equation}\label{E:brq!}
  \qn x_q=q^x-q^{-x}\quad \qn
  {x;n}_q!=\small{\qn x_q\qn{x\!-\!1}_q\cdots\qn{x\!-\!n\!+\!1}_q}
  \quad\qn{n}_q!=\qn{n;n}_q!
  \quad\qbin xk_q\!\!\!\!=\!\frac{\qn{x;k}_q!}{\qn k_q!}.
\end{equation}

\subsubsection{Generic $q$}
Let $q=\e^{h/2}\in\C[[h]]$ and $q^x=\e^{xh/2}$.  Let $U_h(\slt)$ be
the Drinfeld-Jimbo quantization of $\slt$ generated over $\C[[h]]$ by
$H,E,F$ with relations:
\begin{align}
[H,E]=&2E, &[H,F]=-&2F, &[E,F]=&\frac{K-K^{-1}}{q-q^{-1}}.
\end{align}
where $K=q^H$.  Here $U_h(\slt)$ has the $h$-adic topology, i.e. the
topology coming from the inverse limit
$\displaystyle{\lim_{+\oo\leftarrow n}}U_h(\slt)/h^nU_h(\slt)$.  As a
vector space $U_h(\slt)$ is isomorphic to $U(\slt)[[h]]$.
\\
The algebra $U_h(\slt)$  is a Hopf algebra where the coproduct, counit and
antipode are defined by
\begin{align*}
  \Delta(E)&= 1\otimes E + E\otimes K, & \Delta(F)&=K^{-1} \otimes F +
  F\otimes 1, &
  \Delta(H)&=H\otimes 1 + 1 \otimes H,\\ 
  \epsilon(E)&= \epsilon(F)=\epsilon(H)=0, &
  S(E)&=-EK^{-1}, & S(F)&=-KF. 
\end{align*}
In addition, $U_h(\slt)$  is a braided Hopf algebra with $R$-matrix 
$$
R_h=q^{H\otimes H/2} \sum_{n=0}^{\infty} \frac{\{1\}^{2n}}{\{n\}!}q^{n(n-1)/2}
E^n\otimes F^n,
$$

Following the ideas of Habiro\footnote{Habiro uses a isomorphic
  version of $U_q\slt$ obtained by sending $E\mapsto F$, $F\mapsto E$,
  $K\mapsto K^{-1}$ and $q\mapsto v$.} in \cite{Ha} we define $\UqHe$
as the $\Z[q,q^{-1}]$-sub-Hopf algebra of $U_h\slt$ generated by
$H,K,K^{-1},E $ and $\F{n}=\dfrac{\qn1^{2n}}{\qn n!}F^n$, $n\in \N$.
Here the $D$ stands in $\UqHe$ for the divided powers of $F$.
\begin{prop}\label{P:PBW}
  As a $\Z[q,q^{-1}]$-algebra $\UqHe$ is generated by $H,K,K^{-1},E,
  \F{n}$, $n\in \N$ with the relations
\begin{align}
  \label{Eq:0.5}  
  [H,E]=&2E, \qquad KE=q^2EK, \qquad [K,H]=0, \quad [H,\F n]=-2n \F{n},\\
  \label{Eq:1}  
  \F m \F n&=\qbin {m+n}n_q \F {m+n},\quad K\F n=q^{-2n}\F n K,
 \\
  \label{Eq:EF}  
  E^m\F{n}&=\sum_{k=0}^{\min(m,n)}\qbin mk_q\F{n-k}\qn{H-m-n+2k;k}_q!E^{m-k}.
  \end{align}
  Moreover, $\UqHe$ has a PBW type basis over $\Z[q,q^{-1}]$ given by
  basis elements of the form $\F{a}K^bH^cE^d$ for $a,c,d\in~\N$ and
  $b\in \Z$.
\end{prop}
\begin{proof}
  The proof is essentially the same as the proof of Proposition 3.2 in
  \cite{Ha2}: We have that relations \eqref{Eq:0.5}, \eqref{Eq:1} and
  \eqref{Eq:EF} hold in $\UqHe$.  One can prove using only these
  relations that $\UqHe$ is $\Z[q,q^{-1}]$-spanned by the elements
  $\F{a}K^bH^cE^d$ for $a,c,d\in \N$ and $b\in \Z$.  Thus, the
  $\Z[q,q^{-1}]$-algebra is given by these generators and relations.
  Moreover, this proof shows the algebra has the PBW basis as
  described above.
\end{proof}
The Hopf algebra structure of $U_h(\slt)$ induces a Hopf algebra
structure on $\UqHe$, in particular:
\begin{align}
  \label{Eq:3}  
  \Delta(\F n)&=\sum_{k=0}^n q^{k(n-k)}\F{k}K^{k-n}\otimes\F{n-k},\qquad 
  S(\F n)=(-1)^nq^{-n(n-1)}\F nK^{n}.
\end{align}

Remark that the $\F n$ is small in $U_h\slt$ for the $h$-adic topology.

\subsubsection{Specialization of  $q$ to a root of unity}\label{SS:SpecQ}
\newcommand{\f}{\mathsf{f}}
\newcommand{\Qi}{{\mathbb Q(\xi)}}
Fix a positive integer $N$ and let $\xi=e^{\frac{2i\pi}{N}}$ be a
$N^{th}$-root of unity.  Let $r=N/2$ if $N$ is even and $r=N$ if $N$
is odd.  Then $r$ is the smallest positive integer such that $\qn
r_\xi=0$.  If $x\in \C$ then let $\xi^x=\e^{\frac{2ix\pi}{N}}$.  The
number $\xi^{-\frac{r(r-1)}2}$ shows up often in what follows, for
this reason we give it a special notation: let
$\s=\xi^{-\frac{r(r-1)}2}$ which is $1$ if $N$ odd and $i^{1-r}$ if
$N$ is even.
\begin{comment}
Let $x=\xi^\alpha$ then 
$$\qn{\alpha;r}_\xi!=x^{-r}\xi^{-r(r-1)/2}(x^2-1)(x^2-\xi^2)
\cdots(x^2-\xi^{2(r-1)})= x^{-r}\s(x^{2r}-1)=\s\qn{r\alpha}_\xi$$
%
Also for $\ell\in\Z$, we have $\displaystyle{\lim_{q\to\xi}\frac{\qn{\ell
      r}_q}{\qn{r}_q}}=\xi^{r(\ell-1)}\ell$,
$\qn{\alpha+r}_\xi=\xi^r\qn{\alpha}$ and
$$
\qn{r-1}_\xi!=\lim_{\ve\to0}\frac{\qn{r+\ve;r}_\xi!}{\qn{r+\ve}_\xi}
=\lim_{\ve\to0}\s\frac{\qn{r(r+\ve)}_\xi}{\qn{r+\ve}_\xi}
=\lim_{\ve\to0}\s\frac{\qn{r^2+r\ve}_\xi}{\qn{r+\ve}_\xi}
=\lim_{\ve\to0}\s^{-1}\frac{\qn{r\ve}_\xi}{\qn{\ve}_\xi}=\s^{-1}r
$$
\\
%
For $\ell\in\Z$ and $0\le k<r$, we have
$$
\qbin {\ell r+k}k_\xi= \xi^{r\ell k},\quad 
$$
$$
\qbin {(\ell+1) r}r_\xi= 
\qbin {\ell r+r-1}{r-1}_\xi\lim_{q\to\xi}\frac{\qn{(\ell+1) r}_q}{\qn{ r}_q}=
\xi^{r\ell(r-1)}\xi^{r\ell}(\ell+1)=\xi^{r^2\ell}(\ell+1)
$$
\end{comment}

Let $\UxiHe=\UqHe\otimes_{q=\xi}\Qi$ be the specialization of $\UqHe$
at the root of unity $q=\xi$.  The Hopf algebra structure of $\UqHe$
induces a Hopf algebra structure on $\UxiHe$.  Consider the elements
$\f=\F r$ and $F=\frac{\F1}{\qn1}$ of $\UxiHe$.  Then $F^r=\qn r_\xi!
\f/\qn1_\xi^{2r}=0$ in $\UxiHe$.  In $\UxiHe$ the element $K^r$ is
central but $E^r$ does not commute with $H$ and $\f$:
$$[H,E^r]=2rE^r, \;\;\;\;\;
[E,\f]=\F{r-1}\qn{H+1-r}_\xi=
\frac{\s\xi^r\qn1_\xi^{2r-2}}{ r}F^{r-1}\qn{H+1}_\xi, $$
\begin{equation}\label{E:Uwithf}
[E^r,\f]=\qn{H;r}_\xi!=\s\qn{rH}_\xi.
\end{equation}

We will also need two additional algebras.  Let $\Uxi$ be the standard
quantization of $\slt$, i.e. the $\C$-algebra with generators $E, F,
K, K^{-1}$ and the following defining relations:
\begin{equation}\label{E:RelUslt}
  KK^{-1} =K^{-1}K=1,  \,  KEK^{-1} =\xi^2E, \,  KFK^{-1}=\xi^{-2}F,\,
  [E,F] =\frac{K-K^{-1}}{\xi-\xi^{-1}}.
\end{equation}
 Let $\UxiH$ be
the $\C$-algebra given by the generators $E, F, K, K^{-1}, H$,
relations
  \eqref{E:RelUslt}, and the following additional relations:
\begin{align*}
  HK&=KH, & HK^{-1}&=K^{-1}H, & [H,E]&=2E, & [H,F]&=-2F.
\end{align*}
Both $\Uxi$ and $\UxiH$ are sub-Hopf algebras of $\UxiHe$.   

\subsubsection{The category $\catU$ of $\Uxi$-modules}\label{CatOfMod}  
Here we define the category of $\Uxi$-modules which will be used to
define a trace coloring system later in this section.

Let $\ZUo$ be the central sub-algebra of $\Uxi$ generated by the
elements $K^r$ and $E^r$.  Given a $\Uxi$-module $V$ the weight space
corresponding to a \emph{weight} $\kappa\in \C$ is the set of elements
$v\in V$ such that $Kv= \kappa v$.  A $\Uxi$-module $V$ is called a
{\em weight module} if $V$ splits as a direct sum of its weight spaces
and if all elements of $\ZUo$ act diagonally on it.  Let $\catU$ be
the category of finite dimensional weight modules over $\Uxi$.
    
The category $\catU$ has the following simple modules: let $\alpha \in
\C\setminus \Z$ and $\ve\in \C$ then the $r$-dimensional vector space
becomes a simple weight module $W_{\alpha,\ve}$ whose action is
determined by a lowest weight vector $w$ (i.e.  $Fw=0$) such
that $$Kw=q^{\alpha-r+1}w \text{ and } E^rw=\ve w.$$ Since the action
of $E$ is cyclical but the action of $F$ is not, we call the modules
$W_{\alpha,\ve}$ \emph{semi-cyclic}.  Note here we use the middle
weight notation.  For a more complete theory of the representation
theory of $\Uxi$, see \cite{DK} and subsequent papers.

There is a $\Gr$-grading on $\catU$ ($\catU$ fiber over $\Gr$) defined
by the following~: if $g=\mat\kappa\ve\in\Gr$, then
$$\catU_g=\{V\in\catU:F^r,E^r,K^r\text{
  acts by respectively }0,\ve\Id_V,\kappa\Id_V\}.$$ For example the
degree of $W_{\alpha,\ve}$ is given by $\mat{\s^2q^{r\alpha}}{\ve}$.

\subsection{Completion of bigraded bialgebra}
\newcommand{\Ho}{{\mathcal B}} As mentioned above it is the goal of
this section to define a trace coloring system using $\catU$ and the
semi-cyclic modules.  To do this we must define a holonomy braiding
for semi-cyclic modules. To this end, in this subsection we consider a
completion of graded bi-algebras.  In Subsection
\ref{SS:BraidedStructure} we show that this completion leads to a
quasi $R$-matrix in the completion of $(\UxiHe)^{\otimes 2}$.  This is
a general construction because it will be applied to several different
algebras, including the tensor product of certain algebras.  %
 
 Let $\FK$ be a domain.  Let
${\Ho=}\bigoplus_{(w,\ell)\in\Z\times\N} {\Ho}_{w,\ell}$ be a bigraded
$\FK$-algebra with product $\cdot$ and unit $\eta$ such that
\begin{equation}\label{E:gradingunit}
  \eta (\FK)\subset {\Ho}_{0,0},
\end{equation}
\begin{equation}\label{Eq:md}
 {\Ho}_{w,\ell}\cdot{\Ho}_{w',\ell'}\subset 
  \bigoplus_{\ell''=\max(\ell,\ell'-w)}^{\ell+\ell'}{\Ho}_{w+w',\ell''}
\end{equation}
for all $w,w'\in \Z$ and $\ell,\ell'\in \N$.  
If $\Ho$ is a  bialgebra with coproduct $\Delta$ and  counit $\epsilon$, then
we assume that the grading is compatible with the coalgebra maps: 
\begin{equation}\label{E:gradingcounit}
  \epsilon(b)=0, 
  \text{ for all } b\in {\Ho}_{w,\ell}\text{ with }(w,\ell)\neq(0,0),
\end{equation}
\begin{equation}\label{Eq:coprod}
\Delta {\Ho}_{w,\ell}\subset\bigoplus_{  \begin{array}{c}
    w_1+w_2=w\\\ell_1+\ell_2=\ell  \end{array}
}{\Ho}_{w_1,\ell_1}\otimes {\Ho}_{w_2,\ell_2}.
\end{equation}
\begin{comment}
Let ${\Ho}$ be a $\Z$-graded $\FK$-bialgebra (this grading is called
the weight and is compatible with product, coproduct, unit, counit)
with a coalgebra compatible $\N$-grading so that
${\Ho}=\bigoplus_{(w,\ell)\in\Z\times\N} {\Ho}_{w,\ell}$.  The coalgebra
grading means that
$$\Delta {\Ho}_{w,\ell}\subset\bigoplus_{  \begin{array}{c}
    w_1+w_2=w\\\ell_1+\ell_2=\ell  \end{array}
}{\Ho}_{w_1,\ell_1}\otimes {\Ho}_{w_2,\ell_2}$$
We assume that the grading satisfy 
\begin{equation}\label{Eq:md}
  {\Ho}_{w_1,\ell_1}.{\Ho}_{w_2,\ell_2}\subset 
  \bigoplus_{\ell=\max(\ell_1,\ell_2-w_1)}^{\ell_1+\ell_2}{\Ho}_{w_1+w_2,\ell}
\end{equation}
Let say that such a $\Z\times\N$-grading is nice.
\end{comment}
For $w\in \Z$ and $ \ell\in \N$, consider the subspaces
$\FF_{w,\ell}({\Ho})=\bigoplus_{v\le w,k\ge \ell}{\Ho}_{v,k}$.  Then
$$
\FF_{w_1,\ell_1}({\Ho})\cdot\FF_{w_2,\ell_2}({\Ho})\subset
\FF_{w_1+w_2,\max(\ell_1,\ell_2-w_1)}({\Ho}).
$$ 
Define $\wh {\Ho} $ as the direct limit of the inverse limit of the
quotient spaces ${\Ho}/\FF_{w,\ell}({\Ho})$, i.e.
$$\wh {\Ho} =\lim_{w\to+\oo}\lim_{+\oo\leftarrow \ell} {\Ho}/\FF_{w,\ell}({\Ho}).$$
Thus, for each element $\wh u$ of $\wh {\Ho}$ there exists a minimal
$w(\wh u)\in \Z$ such that $\wh u$ is uniquely written as a sum
$\sum_{\ell=0}^{+\oo}u_\ell$ where $u_\ell\in
\bigoplus_{w=-\oo}^{w(\wh u)}{\Ho}_{w,\ell}$.  Equation \eqref{Eq:md}
implies that the
multiplication of ${\Ho}$ extend to $\wh {\Ho}$.  In particular, if
$\wh u',\wh u''\in \wh {\Ho}$ then for each $\ell\in \N$ only finitely
many terms contribute to $\bp{\wh u'\cdot\wh u''}_\ell$ and all terms
of the product $\wh u'\cdot\wh u''$ have $\Z$-degrees smaller that
$w(\wh u')+w(\wh u'')$.

The assignment $\Ho\mapsto \wh {\Ho}$ is functorial.  In particular,
if $\Ho'$ is a $\FK$-bialgebra graded as above then each algebra
morphism $f:{\Ho}\to \Ho'$ preserving the grading induces an algebra
morphism $\wh f:\wh {\Ho}\to\wh \Ho'$.  Also, for $n>0$,
${\Ho}^{\otimes n}$ inherit a $\Z\times\N$-grading from ${\Ho}$ (by
summing the degrees of the factors).  This grading satisfies the
relations given in \eqref{Eq:md}.  Consider the completion
$\displaystyle{\wh{{\Ho}^{\otimes n}}}$.  Then, the functoriality
implies that the coproduct and counit on ${\Ho}$ naturally induces a
coalgebra structure on $\wh{\Ho}$:
$$\wh \Delta:\wh {\Ho}\to\wh{{\Ho}^{\otimes 2}},\;\;\; 
\wh{\epsilon}: \wh {\Ho}\to\wh{\FK}$$
where $\wh{\FK}=\FK$ has the trivial grading.

\begin{prop}
  For $w\in \Z$ and $\ell\in \N$, let ${(\UqHe)}_{w,\ell}$ be the free
  $\Z[q,q^{-1}]$-module with basis
  $$\{\F aK^bH^cE^d : a,c,d\in \N, b\in \Z,
  w=2d-2a \text{ and } \ell=2a\}.$$ Then
  $\UqHe=\bigoplus_{(w,\ell)\in\Z\times\N} {(\UqHe)}_{w,\ell}$ is a
  bigrading on $\UqHe$.  Moreover, this bigrading is preserved in
  $\UxiHe$ after specializing $q=\xi$.  Both of these bigradings
  satisfy relations \eqref{E:gradingunit}--\eqref{Eq:coprod} and so
  there exists bialgebras $ \wh{\UqHe}$ and $\wh{\UxiHe}$.
\end{prop}
\begin{proof}
  Equations \eqref{E:gradingunit} and \eqref{E:gradingcounit} are
  satisfied, since the Hopf algebra structure of $\UxiHe$ is induced
  by $U_h(\slt)$.  Moreover, it is easy to see from the definition of
  the coproduct that Equation \eqref{Eq:coprod} is satisfied.
  Finally, Equation \eqref{Eq:md} follows from using Equation
  \eqref{Eq:EF} applied to the product of two elements of the PBW base
  $\{\F aK^bH^cE^d\}_{a,c,d\in \N, b\in \Z}$ of $\UqHe$.
\end{proof}

To give some feeling for the completions $ \wh{\UqHe}$ and
$\wh{\UxiHe}$ we consider the following example.  The subspace
$\FF_{-6,2n}({\UqHe})$ is linear combinations of elements of the form
$\F a K^bH^cE^d$ where $a\ge n$ and $d\le a-3$.  So
${\UqHe}/\FF_{-6,2n}({\UqHe})$ contains the element
$c_{n-1}\F{n-1}E^{n-4}+c_{n-2}\F{n-2}E^{n-5}+\cdots+ c_{3}\F{3}E^{0}$
where $c_i\in \Z[q,q^{-1}]$.  Also, note that $\F{n}E^{n-3}=0$ in
${\UqHe}/\FF_{-6,2n}({\UqHe})$.  Thus, $\lim_{+\oo\leftarrow \ell}
{\UqHe}/\FF_{-6,\ell}({\UqHe})$ and $\wh {\UqHe}$ contain the element
$\sum_{i=3}^\oo c_i\F{i}E^{i-3}$.  Also remark that for any $\wh u\in
\FF_{0,1}({\UxiHe})$ and any formal power series $g(X)\in\C[[X]]$, the
expression $g(\wh u)$ is a convergent series in $ \wh{\UxiHe}$.

\subsection{R-matrix}\label{SS:BraidedStructure} In this subsection
we show that $\wh{\UqHe}$ admits a category of modules which is braided.
\begin{lem}
  The inclusion $\UqHe\to U_h(\slt)$ extends uniquely to an injective 
  bialgebra morphism
  $$\wh{\UqHe}\rightarrow U_h(\slt).$$
\end{lem}
\begin{proof}
 Since $\F \ell\in h^{\ell}U_h(\slt)$ we have that $\FF_{w,2\ell}(\UqHe)\subset
  h^{\ell}U_h(\slt)$.   Thus, the series corresponding to  $\wh u\in \wh{\UqHe}$ 
  converges to a well defined element  in the $h$-adic topology of $U_h(\slt)$.
\end{proof}
\newcommand{\HH}{{\mathcal H}}
Let $\HH: U_h(\slt)^{\otimes 2}\to U_h(\slt)^{\otimes 2}$ be the automorphism
given by conjugation by $q^{H\otimes H/2}$.  Hence if $x,y\in U_h(\slt)$
satisfy $[H,x]=2mx$ and $[H,y]=2ny$ then 
$$\HH(x\otimes y)=q^{2mn}xK^n\otimes yK^m.$$
Furthermore, with usual notation, we have $(\Id\otimes\Delta)(\HH(x\otimes
y))=\HH_{12}\circ\HH_{13}(x\otimes (\Delta y))$ and
$(\Delta\otimes\Id)(\HH(x\otimes y))=\HH_{13}\circ\HH_{23}((\Delta x)\otimes
y)$. 

Consider the element $\check R=q^{-H\otimes H/2}R_h$ of
$U_h(\slt)^{\hat\otimes 2}$.  The defining properties of the R-matrix
$R_h$ imply the following relations for $\check R$:
\begin{equation}
  \label{eq:Dop}
  \HH\circ\Ad_{\check R}\circ\Delta=\Delta^{op},
\end{equation}
\begin{equation}
  \label{eq:R12}
  (\Id\otimes\Delta)\check R=\HH_{13}^{-1}(\check R_{12})\check R_{13},
\end{equation}
\begin{equation}
  \label{eq:R23}
  (\Delta\otimes\Id)\check R=\HH_{23}^{-1}(\check R_{13})\check R_{23}.
\end{equation}
The automorphism $\HH$ restrict to an automorphism of the algebra
$(\UqHe)^{\otimes 2}$ such that
 $$\HH(\FF_{w,\ell}((\UqHe)^{\otimes
   2}))\subset\FF_{w,\ell}((\UqHe)^{\otimes 2})$$ for all
 $(w,\ell)\in\Z\times\N$.  Hence $\HH$ extends to an automorphism of the
 completion $\wh{(\UqHe)^{\otimes 2}}$.  This automorphism specializes to an
 automorphism of $\wh{(\UxiHe)^{\otimes 2}}$.

\begin{theo}\label{T:3R-matrix} 
Let $\check R_t=\sum_{n=0}^{r-1} q^{n(n-1)/2}E^{n}\otimes \F n$. Then the elements 
  \begin{align*}
    \check R_q=&\sum_{n=0}^\oo q^{n(n-1)/2}E^{n}\otimes \F n
    \in\wh{(\UqHe)^{\otimes 2}}\\
   \check R_\xi=&\check R_t\exp\bp{\s^{-1}E^r\otimes \f}
    =\exp\bp{\s^{-1}E^r\otimes \f}\check R_t\in\wh{(\UxiHe)^{\otimes 2}}
  \end{align*}
satisfy Relations \eqref{eq:Dop}, \eqref{eq:R12} and  \eqref{eq:R23} where $\f=\F r\in \wh{\UxiHe}$.
\end{theo}
\begin{proof}
  The image of the element $\check R_q$ in $U_h(\slt)\wh\otimes
  U_h(\slt)$ is equal to $\check R$.  Thus, it follows that $\check
  R_q$ satisfies Relations \eqref{eq:Dop}--\eqref{eq:R23} since
  $\check R$ does.  Moreover, these relations can be specialized at
  $q=\xi$.  In particular, the specialization has the following
  properties:
  $$\qn{x+\ell r}_q\mapsto \xi^{\ell r}\qn{x}_\xi,\quad
  \frac{\qn{r\ell}_q}{\qn{r}_q}\mapsto\ell \xi^{r\ell},\quad \F {\ell
    r+k}\mapsto \frac{\xi^{\frac{\ell(\ell-1)r^2}2+\ell
      rk}\qn1_\xi^{2k}}{\qn k_\xi!\,\ell!}\,\f^\ell\, F^k\text{ for }
  \ell,k\in\N,\,k<r$$ Finally, the image of $\check R_q$ by the
  morphism of bialgebra $\wh{(\UqHe)^{\otimes
      2}}\to\wh{(\UxiHe)^{\otimes 2}}$ is given by
  $$
  \check R_\xi=\sum_{\ell=0}^\oo\sum_{n=\ell
    r+k,\,k=0}^{r-1}\xi^{\frac{k(k-1)+\ell r(\ell r-1)}2+\ell r k}
  \xi^{\frac{\ell(\ell-1)r^2}2+\ell rk} \frac{\qn1_\xi^{2k}}{\qn
    k_\xi!\,\ell!}E^k\otimes F^k(E^r\otimes\f)^\ell,
  $$
  and so
  $$
  \check R_\xi=\check R_t \sum_{\ell=0}^\oo\xi^{\frac{\ell r(\ell
      r-1)+r^2\ell(\ell-1)}2}
  \frac{1}{\ell!}(E^r\otimes\f)^\ell=\check
  R_t\exp(\s^{-1}E^r\otimes\f)
  $$
  since $\xi^{r^2\ell^2}=\xi^{r^2\ell}$ implies $\xi^{\frac{\ell
      r(\ell
      r-1)+r^2\ell(\ell-1)}2}=\xi^{\ell\frac{r(r-1)}2}=\s^{-\ell}$.
  \\
\end{proof}

Recall that in Subsection \ref{SS:SpecQ} we considered the element
$\f=\F r\in \wh{\UxiHe}$.  A $\wh{\UxiHe}$-module $\wh V$ is
$\f$-profinite if $\wh V/\f^\ell \wh V$ is finite dimensional for all
$ \ell\in\N$ and
$$\wh V\simeq\lim_{+\oo\leftarrow\ell}\wh V/\f^\ell \wh V.$$
Let $\wh\cat$ be the category of $\f$-profinite $\wh{\UxiHe}$-modules.  
Then $\wh\cat$ is a monoidal category where the tensor product is given by
 $$\wh V\wh\otimes\wh W=\lim_{+\oo\leftarrow\ell}(\wh V/\f^\ell \wh
V)\otimes(\wh W/\f^\ell \wh W).$$
\begin{theo}\label{T:compBraiding}
  The category $\wh \cat$ is braided with braiding $c_{\wh V,\wh W}:
  \wh V\wh\otimes\wh W\to\wh W\wh\otimes\wh V$ given by
  $$c_{\wh V,\wh W}(v\otimes
  w)=\tau(R_\xi(v\otimes w))$$
   where $R_\xi=\xi^{H\otimes H/2}\check R_\xi$ and $\tau:\wh V\wh\otimes\wh W\to\wh
  W\wh\otimes\wh V$ is the linear flip map that exchange the two factors.
\end{theo}
\begin{proof}
  Since $H\f^\ell=\f^\ell H-2r\ell \f^\ell$ we have the action of $H$
  stabilize $\f^\ell \wh V$.  Thus $H$ induces a well defined bounded
  operator on $\wh V/\f^\ell \wh V$.  Therefore,
  $\dfrac{i\pi}NH\otimes H$ is a bounded linear map on the finite
  dimensional $\C$-vector space $(\wh V/\f^\ell \wh V)\otimes(\wh
  W/\f^\ell \wh W)$ and its exponential $\xi^{H\otimes H/2}$ is a
  convergent series.  These operators commute with the natural maps
  $(\wh V/\f^{\ell+1} \wh V)\otimes(\wh W/\f^{\ell+1} \wh W)\to (\wh
  V/\f^\ell \wh V)\otimes(\wh W/\f^\ell \wh W)$ and this implies that
  the action of $\xi^{H\otimes H/2}$ is well defined on $\wh
  V\wh\otimes\wh W$.

  Let $\wh V$ and $\wh W$ be objects in $\wh \cat$.  Let
  $\rho:\wh{\UxiHe}\wh\otimes \wh{\UxiHe} \to\End_{\wh \cat}(\wh
  V\wh\otimes\wh W)$ be the representation map corresponding to the
  module structure of $\wh V\wh\otimes\wh W$.  Then for any
  $x\in\wh{\UxiHe}\wh\otimes \wh{\UxiHe}$, we have $\xi^{H\otimes
    H/2}\rho(x)\xi^{-H\otimes H/2}=\rho(\HH(x))$.  Combining this with
  Equations \eqref{eq:Dop}, \eqref{eq:R12} and \eqref{eq:R23} for
  $\check R_\xi$ implies that $R_\xi$ is an R-matrix operator for the
  modules in $\wh \cat$.  Thus, $\wh \cat$ is braided where the
  braiding given by the action of $\check R_\xi$ composed with the
  flip map.
\end{proof}

\subsection{Holonomy braiding} In this subsection we will show that a
semi-cyclic module can be realized as a submodule of a certain
$\f$-profinite $\wh{\UxiHe}$-module.  Then we show that the braiding
in $\wh \cat$ induces the desired holonomy braiding on such
submodules.

  Let $\E =\C[[\f]]$ be the $\Uw$-module whose action is given by 
$$E=F=0,\quad K=\Id,\quad \f=\f.\Id,\quad H=-2r\f\frac d{d\f}.$$
By a \emph{nilpotent} $\UxiH$-module $V$, we mean a $\UxiH$-module
such that $E^r$ and $F^r$ both act by 0 on $V$, $q^H$ acts as $K$ on
$V$ and $V$ splits as a direct sum of its $H$-weight space.  Let
$\cat^H$ be the category of nilpotent $\UxiH$-modules (for more
details see \cite{GPT}).

Recall that $\Uxi\subset \UxiH\subset\UxiHe\subset\Uw$.  If $V$ is a
nilpotent $\UxiH$-module, let $\E V=\Uw\otimes_{\UxiH}V$ be the
infinite dimensional induced $\Uw$-module.  Proposition \ref{P:PBW}
(i.e. the PBW theorem) implies that $\E V=V[[\f]]$ and $\E V/\f^\ell\E
V\simeq\bigoplus_{n=0}^{\ell-1}\f^nV$.  Therefore, $\E V$ is a module
of $\wh \cat$.  Also, if $\C$ is the trivial nilpotent $\UxiH$-module
then $\E=\E\C$.  Thus, the induction $\Uw\otimes_{\UxiH}\cdot$ gives a
functor $F_\E:\cat^H\to\wh\cat.$ However, this functor is not monoidal
because $\E (V\otimes W)\neq \E V\wh\otimes \E W$ and
$F_\E(\unit)=\E$.

If $V,W$ are nilpotent $\UxiH$-modules then $V\otimes W$ is a
sub-$\UxiH$-module of the $\Uw$-module $\E V\wh\otimes\E W$.  This
space is stable by the multiplication by $R_\xi$ because $E^r=0$ on
$W$ and on this space $R_\xi$ acts as $\xi^{H\otimes H/2}\check R_t$
which is the truncated $R$-matrix used to define a braiding on
nilpotent $\UxiH$-modules (see \cite{GPT}).

\begin{theo} \label{T:NilpTwisted} Let $V$ be a nilpotent
  $\UxiH$-module. Let $W_{\alpha,\ve}$ be the semi-cyclic
  $\Uxi$-module corresponding to $\alpha \in \C\setminus \Z$ and
  $\ve\in \C$ (see Subsection \ref{CatOfMod}).  If $t\in \C$ then the
  subspace $$V^t=\exp(-t\f)V\subset\E V$$ is a finite dimensional
  $\Uxi$-sub-module of $\E V$.  In addition, if $t$ is the complex
  number determined by $t=\frac{-\s\ve}{\qn{r\alpha}}$ then the
  $\Uxi$-modules $(V_\alpha)^t$ and $W_{\alpha,\ve}$ are isomorphic,
  where $V_\alpha$ is the simple nilpotent $\UxiH$-module with a
  lowest weight vector $v_0$ satisfying $Fv_0=0$ and
  $Kv_0=q^{\alpha-r+1}v_0$.
\end{theo}
\begin{proof}
  First, the subspace $V^t$ is finite dimensional with 
  dimension $\dim(V)$.  Also, $V^t$ is a $\Uxi$-module since
  Equations \eqref{E:Uwithf} imply
  $\Uxi\exp(-t\f)=\exp(-t\f)\Uxi\subset\Uw$.  For example,
  $E^r\exp(-t\f)=\exp(-t\f)(E^r-\s t\qn{rH})$.  Moreover, the last
  equation determines the action of $E^r$ on $V^t$:
  $$E^r\exp(-t\f)v=\exp(-t\f)(E^r-\s t\qn{rH})v=-\exp(-t\f)\s t\qn{rH}v$$ 
  for any $v\in V^t$.  So if $V^t=(V_\alpha)^t$ with
  $t=\frac{-\s\ve}{\qn{r\alpha}}$ then $E^rv=\ve v$ for all $v\in
  (V_\alpha)^t$.  Let $v_0$ be the lowest weight vector of $V_\alpha$
  as in the statement of the theorem.  Then $\exp(-t\f)v_0$ is a
  lowest weight vector in $(V_\alpha)^t$ since $F \exp(-t\f)v_0 =
  \exp(-t\f)Fv_0 =0$ and
  $K\exp(-t\f)v_0=\exp(-t\f)Kv_0=q^{\alpha-r+1}\exp(-t\f)v_0$.  Thus,
  $\exp(-t\f)v_0$ generates a $r$-dimensional module which is equal to
  $ (V_\alpha)^t$ and isomorphic to $W_{\alpha,\ve}$.
  \end{proof}
  One should note that if $t\neq0$ then $V^t$ is not a $\UxiH$-module
  since $HV^t\not\subset V^t$ for $t\neq0$.  This fact is an important
  aspect of the existence of the non-trivial holonomy $R$-matrix.  In
  particular, the following lemma shows that the map $\xi^{H\otimes
    H/2}$ ``moves'' the tensor product of two of the
  $\Uxi$-sub-modules of Theorem \ref{T:NilpTwisted}.

\begin{lem}\label{L:ActHtimeH}
  Let $V_1$ and $V_2$ be nilpotent $\UxiH$-modules on which $K^r$ acts
  by $\k_1,\k_2$ respectively.  Then on $V_1\otimes V_2\subset\E
  V_1\otimes\E V_2$, we have
  $$\xi^{H\otimes H/2}(\exp(-t_1\f)\otimes\exp(-t_2\f))
  =(\exp(-t_1\k_2^{-1}\f)\otimes\exp(-t_2\k_1^{-1} \f))\xi^{H\otimes H/2}.$$ In
  particular, the $\Uxi$-module $V_1^{t_1}\otimes V_2^{t_2}$ is sent
  into the $\Uxi$-sub-module $V_1^{\k_2^{-1} t_1}\otimes V_2^{\k_1^{-1}
    t_2}\subset\E V_1\otimes\E V_2$.\\
  Similarly, on $V_1\otimes V_2$, we have
  $$\exp\bp{\s^{-1} E^r\otimes \f}(\exp(-t_1\f)\otimes\exp(-t_2\f))
  =\bp{\exp(-t_1\f)\otimes1}\bp{\exp\bp{-\bp{t_1\qn{rH}+t_2}\otimes\f}}
$$
$$=\exp(-t_1\f)\otimes\exp(-\bp{(\k_1-1/\k_1)t_1+t_2}\f).$$
\end{lem}
\begin{proof}
We have
$$\xi^{H\otimes H/2} (\f^n\otimes 1)=(\f^n\otimes 1)\xi^{(H-2rn)\otimes H/2} = (\f^n\otimes 1)(1\otimes K^{-r})^n \xi^{H\otimes H/2}= (\f\otimes  K^{-r})^n \xi^{H\otimes H/2}.$$
Similarly, $\xi^{H\otimes H/2} (1\otimes \f^n)= (K^{-r}\otimes  \f)^n \xi^{H\otimes H/2}.$
Thus, these equalities imply the first relation of the lemma.  A similar argument implies the second relation of the lemma.
\end{proof}
\begin{theo}\label{Th:holonomyR}
  The restriction of $R_\xi$ on $V_1^{t_1}\otimes V_2^{t_2}$ is a map 
  $$V_1^{t_1}\otimes V_2^{t_2}\to V_1^{t'_1}\otimes V_2^{t'_2}$$ 
  where $t'_1=\k_2^{-1}t_1$ and $t'_2=(1-\k_1^{-2})t_1+\k_1^{-1}t_2$.
  In particular, if $a_1=\exp(-t_1\f)v_1 \in V_1^{t_1}$ and $a_2=\exp(-t_2\f)v_2\in V_2^{t_2}$ then 
  \begin{align}\label{E:ActionOfRmatrix}
    R_\xi(a_1\otimes a_2)
    &=(\exp(-t'_1\f)\otimes\exp(-t'_2\f))\xi^{H\otimes H/2}
    \sum_{n=0}^{r-1}\frac{\qn1^{2n}}{\qn n_\xi!}\xi^{n(n-1)/2}
    \Ad_{t_1}(E)^n v_1\otimes F^n v_2
  \end{align}
  where $\Ad_t(x)=\exp(t\f)x\exp(-t\f)$.
\end{theo}
\begin{proof}
  Recall $R_\xi=\xi^{H\otimes H/2}\check R_\xi$ where $\check
  R_\xi=\check R_t\exp\bp{\s^{-1}E^r\otimes \f}$ by Theorem
  \ref{T:3R-matrix}.  Then the theorem follows from the following
  computation for maps restricted on $V_1\otimes V_2$
  \begin{align*}
    R_\xi(\exp(-t_1\f)\otimes\exp(-t_2\f))
    &= \xi^{H\otimes H/2}\check R_t\exp\bp{\s^{-1}E^r\otimes \f}(\exp(-t_1\f)\otimes\exp(-t_2\f))\\
    &= \xi^{H\otimes H/2}\check R_t\bp{\exp(-t_1\f)\otimes\exp(-\bp{(\k_1-1/\k_1)t_1+t_2}\f)}\\
     &= \xi^{H\otimes H/2}\bp{\exp(-t_1\f)\otimes\exp(-\bp{(\k_1-1/\k_1)t_1+t_2}\f)}(\Ad_{t_1}\otimes\Id)(\check R_t)\\
    &=(\exp(-t'_1\f)\otimes\exp(-t'_2\f))\xi^{H\otimes H/2}
    (\Ad_{t_1}\otimes\Id)(\check R_t)
  \end{align*}
  where the second and fourth equalities follow from Lemma \ref{L:ActHtimeH}.
\end{proof}
Let $V_1,V_2$ be nilpotent $\UxiH$-modules.  Let $t_1, t_2\in \C$ and set $t'_1=\k_2^{-1}t_1$ and
  $t'_2=\k_1^{-1}t_2+(1-\k_1^{-2})t_1$ as above.  
 Recall the braiding $c_{\wh V,\wh W}$ of $\wh\cat$ given in Theorem \ref{T:compBraiding}.   Theorem \ref{Th:holonomyR} implies that the restriction of this braiding to 
  $V_1^{t_1}\otimes V_2^{t_2}\subset\E V_1\otimes\E V_2$ gives a $\Uxi$-module map
  \begin{equation}\label{E:DefholonomyB}
  C_{V_1^{t_1},V_2^{t_2}}:V_1^{t_1}\otimes V_2^{t_2}\to V_2^{t'_2}\otimes V_1^{t'_1}
  \end{equation}
 such that the set of such maps satisfy the braid relations.  
  
 Furthermore, if $x,y,x',y'$ are the degree in $\Gr$ of respectively
 $V_1^{t_1}, V_2^{t_2},V_1^{t'_1}, V_2^{t'_2}$ (see Theorem
 \ref{T:NilpTwisted}), one easily gets that $(x',y')=\RY(x,y)$ where
 $\RY$ is the map of section \ref{SS:ExG}.
\subsection{The pivotal structure and the right trace on the ideal of
  projectives}\label{SS:PivTrace}
Recall the definition of the category $\catU$ given in Subsection
\ref{CatOfMod}.

The category $\catU$ is a pivotal $\C$-category where for any 
object $V$ in $\catU$, the dual object and the duality
morphisms are defined as follows: $V^* =\Hom_\C(V,\C)$ and
\begin{align}\label{E:DualityForCat}
  \coev_{V} :\, & \C \rightarrow V\otimes V^{*} \text{ is given by } 1 \mapsto
  \sum
  v_j\otimes v_j^*,\notag\\
  \ev_{V}:\, & V^*\otimes V\rightarrow \C \text{ is given by }
  f\otimes w \mapsto f(w),\notag\\
  \tev_{V}:\, & V\otimes V^{*}\rightarrow \C \text{ is given by } v\otimes f
  \mapsto f(K^{1-{\ro}}v),\notag
  \\
  \tcoev_V:\, & \C \rightarrow V^*\otimes V \text{ is given by } 1 \mapsto \sum
  v_j^*\otimes K^{{\ro}-1}v_j,
\end{align}
where $\{v_j\}$ is a basis of $V$ and $\{v_j^*\}$ is the dual basis of $V^*$.
Let $\Proj$ be the full subcategory of $\catU$ consisting of projective
$\Uxi$-modules. Then $\Proj$ is an ideal.

\begin{prop}
  Let $V_0$ be the unique projective irreducible weight module on
  which $K^r-1$ and $E^r$ vanish.  Then $V_0$ is right ambi.
\end{prop}
\begin{proof}
  Let $\catU_0$ be the full subcategory of $\catU$ consisting of
  modules on which $K^r-1$ and $E^r$ vanish.  Then $\catU_0$ is the
  category of finite dimensional modules of the ``small quantum
  group'' $u_\xi$ which is the finite dimensional quotient of $\Uxi$
  by the ideal generated by $K^r-1$ and $E^r$.  It is well known that
  $u_\xi$ is unimodular so by \cite[Lemma 4.2.1 and Corollary
  3.2.1]{GKP2}, its simple module $V_0$ is right ambi.
\end{proof}
Let $\t$ be the corresponding right trace on $\Proj=\ideal_{V_0}$ (see
Theorem \ref{T:traceXtraceI}).  The following proposition give an
explicit formula for the trace $\t$.
\begin{prop}
  For any $f\in\End_\catU(V_\alpha^t\otimes W)$, we have
  $$\t(f)=\dfrac{\qn\alpha}{\qn{r\alpha}}
  \tr_\C((\Id_{V_\alpha^t}\otimes K^{1-r})\circ f).$$
  In particular, the modified dimension of $V_\alpha^t$ is the scalar 
  $\t(\Id_{V_\alpha^t})=\dfrac{r\qn\alpha}{\qn{r\alpha}}.$
\end{prop}
\begin{proof}
  Let $f\in\End_\catU(V_\alpha^t\otimes W)$.  Since $V_\alpha^t$ is
  simple there exists a scalar $\lambda$ such that
  $\tr_r^{W}(f)=\lambda \Id_{V_\alpha^t}$.  Then the properties of the
  trace imply that
  \begin{equation}\label{E:tf}
  \t(f)=\t(\lambda \Id_{V_\alpha^t})=\lambda\t(\Id_{V_\alpha^t}).
  \end{equation}
  On the other hand, by the definition of $\tr_r^{W}$ we have
  $\tr_\C((\Id_{V_\alpha^t}\otimes K^{1-r})\circ f)=\tr_\C(\lambda
  \Id_{V_\alpha^t})=r\lambda$ (the appearance of $K^{1-r}$ comes from
  the morphism $\tev_W$ used in the definition of $\tr_r^{W}$). Thus,
  Equation \ref{E:tf} implies that the proposition follows from the
  computation of $\t(\Id_{V_\alpha^t})$.  The braiding
  \eqref{E:DefholonomyB} gives maps
  $$C_{V_\alpha^t,V_0}:V_\alpha^t\otimes V_0\to V_0\otimes V_\alpha^t\text{ and }
  C_{V_0,V_\alpha^t}:V_0\otimes V_\alpha^t\to V_\alpha^t\otimes V_0.$$
  The right partial trace of $C_{V_\alpha^t,V_0}\circ
  C_{V_0,V_\alpha^t}$ and $C_{V_0,V_\alpha^t}\circ C_{V_\alpha^t,V_0}$
  are scalar endomorphisms.  A computation similar to the proof of
  Theorem \ref{T:tcsys} on highest weight vectors leads to~:
  $$\tr_r^{V_\alpha^t}(C_{V_\alpha^t,V_0}\circ  C_{V_0,V_\alpha^t})=r \Id_{V_0}
  \text{ and }\tr_r^{V_0}(C_{V_0,V_\alpha^t}\circ C_{V_\alpha^t,V_0})=
  \frac{\qn{r\alpha}}{\qn{\alpha}} \Id_{V_\alpha^t}.$$ As $\t$ has the
  same value on these two morphisms and by definition
  $\t(\Id_{V_0})=1$, we get the above formula for
  $\t(\Id_{V_\alpha^t})$.
\end{proof}
\subsection{The trace coloring system}
In this subsection we will show that $\catU$ has a trace coloring system.  
Let
$$I=\{{\j}:\C\setminus\{-1,0,1\}\to\C\text{ such that }
\exp\circ\, {\j}=\Id_{\C\setminus\{-1,0,1\}}\}$$ be a set of inverse
functions of the exponential.  Let $(\Gd, \wb \Gd, \Gr, \vp_+,\vp_-)$,
$\Yd$, $\Y$ and $\RY$ be as in Section \ref{SS:ExG}.  We identify $\Y$
with the set
$$\Y\simeq\{(\kappa,\ve)\in(\C\setminus\{-1,0,1\})\times\C\}.$$

For ${\j}\in I$ and $(\kappa,\ve)\in \Y$ define
$A^{\j}_{(\kappa,\ve)}$ as the $\Uxi$-module $V_\alpha^t$ given in
Theorem \ref{T:NilpTwisted} where $\alpha=\dfrac{N
  {\j}(\s^2\kappa)}{2i\pi r}$ and
$t=\dfrac{-\ve}{\s(\kappa-\kappa^{-1})}$.  Then the map in Equation
\eqref{E:DefholonomyB} defines a holonomy braiding
$\{\B_{y,z}^{\j,\j'}\}_{y,z\in \Y, \; \j,\j'\in I}$ on the modules
$A^{\j}_{(\kappa,\ve)}$.  Finally, let $\t$ be the right trace given
in Subsection \ref{SS:PivTrace}.
\begin{theo}\label{T:tcsys}
  The tuple $(\Y,\RY,\{A^{\j}_y\},\{\B_{y,z}^{\j,\j'}\},\t)$ is a
  \emph{trace coloring system} in $\catU$ where the the twist element
  for ${\j}\in I, y=(\kappa,\ve)\in \Y$ is given by
  $$\theta^{\j}_{(\kappa,\ve)}=q^{\frac{\alpha^{2}-(r-1)^{2}}2}$$
  where $\alpha=\dfrac{N {\j}(\s^2\kappa)}{2i\pi r}.$  
\end{theo}
\begin{proof}
  First, for $(\kappa,\ve) \in \Y$, the map $\shift$ of Proposition
  \ref{P:G-Markov} is given by
  $\shift((\kappa,\ve))=(\kappa,\ve\kappa^{-1})\in \Y$.

  Next, we need to show that we have a $\Gr$-Markov trace as defined
  in Section \ref{GMarkovTraces}.  In particular, we must define the
  twist element $\theta_y^{\j}$ for ${\j}\in I, y\in \Y$.  To do this
  we consider the endomorphism
\begin{equation}\label{E:DefOff}
f=\tr_r^{A^{\j}_{\shift(y)}}(\B_{y,\shift(y)}^{{\j},{\j}})\in \End(A^{\j}_y).
\end{equation}
Since $A^{\j}_y$ is simple there exists a constant $\theta_y^{\j}$
such that $f(a)=\theta_y^{\j} a$ for all $a\in A^{\j}_y$.  We will
compute $\theta_y^{\j}$ directly and show that it satisfies the
properties of the twist.

Let ${\j}\in I$ and $y=(\kappa,\ve)\in \Y$.  
As above set $\alpha=\dfrac{N {\j}(\s^2\kappa)}{2i\pi r}$ and
$t=\dfrac{-\ve}{\s(\kappa-\kappa^{-1})}$.  Let $V_\alpha$ be the
nilpotent $\UxiH$-module as in Theorem \ref{T:NilpTwisted}.  Let $v_0$
be a highest weight vector of $V_\alpha$, i.e. $Ev_0=0$ and
$Hv_0=(\alpha +r-1)v_0$.  Let $\{v_i\}$ be a basis of $V_\alpha$ with
$v_i=v_0$ for $i=0$.  Then $\{a_i=\exp(-\kappa^{-1}t\f)v_i\}$ is a
basis of $(V_\alpha)^{\kappa^{-1} t}$.

We have
\begin{align*}
f(\exp(-t\f)v_0) &= \tr_r^{A^{\j}_{\shift(y)}}\left(\B_{y,\shift(y)}^{{\j},{\j}}\right)\left(\exp(-t\f)v_0\right)\\
&=\tr_r^{V^{\kappa^{-1}t}_\alpha}\left(\B_{V^{t}_\alpha,V^{\kappa^{-1}t}_\alpha}\right)\left(\exp(-t\f)v_0\right) \\
&= \left(\Id_{V^{t}_\alpha} \otimes   \tev_{V^{\kappa^{-1}t}_\alpha}\right)\left(\B_{V^{t}_\alpha,V^{\kappa^{-1}t}_\alpha}\otimes \Id_{(V^{\kappa^{-1}t}_\alpha)^*}\right)\left(\Id_{V^{t}_\alpha}\otimes \coev_{V^{\kappa^{-1}t}_\alpha}\right)\left(\exp(-t\f)v_0\right)\\
&= \left(\Id_{V^{t}_\alpha} \otimes   \tev_{V^{\kappa^{-1}t}_\alpha}\right)\left(\B_{V^{t}_\alpha,V^{\kappa^{-1}t}_\alpha}\otimes \Id_{(V^{\kappa^{-1}t}_\alpha)^*}\right)\left(\exp(-t\f)v_0\otimes \sum_i a_i\otimes a^*_i\right)\\
&= \left(\Id_{V^{t}_\alpha} \otimes   \tev_{V^{\kappa^{-1}t}_\alpha}\right)\left(\tau R_\xi  \otimes \Id_{(V^{\kappa^{-1}t}_\alpha)^*}\right)\left(\exp(-t\f)v_0\otimes \sum_i \exp(-\kappa^{-1}t\f)v_i \otimes a_i^*\right)\\
&= \left(\Id_{V^{t}_\alpha} \otimes   \tev_{V^{\kappa^{-1}t}_\alpha}\right)\left(\tau \left(\exp(-\kappa^{-1}t\f)\otimes \exp(-t\f)\xi^{H\otimes H/2} v_0\otimes v_0 \right)\otimes a_0^*\right)\\
&=q^{\left(\alpha + r-1\right)^2/2 +(1-r)\left(\alpha+r-1\right)}\exp(-t\f)v_0
\end{align*}  
where the first five equalities are definitions, the sixth equality
comes from Equation \eqref{E:ActionOfRmatrix} combined with the fact that
$Ev_0=0$ and the final equality comes from:
$$ \tev_{V^{\kappa^{-1}t}_\alpha}(\exp(-\kappa^{-1}t\f)v_0\otimes a_0^*)
=a_0^*((\exp(-\kappa^{-1}t\f)K^{1-r}v_0))=q^{(1-r)(\alpha +r
  -1)}a_0^*(a_0)=q^{(1-r)(\alpha +r -1)}.$$

Now if $\psi(\kappa,\ve)$ and $\psi(\kappa',\ve')$ are conjugate in
$\wb \Gd$ then $\kappa=\kappa'$ so for any ${\j}\in I$, let
$\alpha=\dfrac{N {\j}(\s^2\kappa)}{2i\pi r}=\dfrac{N
  {\j}(\s^2\kappa')}{2i\pi r}$, and we have
$\theta^{\j}_{(\kappa,\ve)}=q^{\frac{\alpha^{2}-(r-1)^{2}}2}=\theta^{\j}_{(\kappa',\ve')}$.

Finally, let
$k=(\Id\otimes(\B_{y,\shift(y)}^{{\j},{\j}})^{\pm1})(h\otimes\Id)$
then by definition of a right trace we have
$$\tt_{V\otimes A_y^{\j}\otimes A_{\shift(y)}^{\j}}(k)
  =\tt_V\left(\tr_r^{A_y^{\j}\otimes A_{\shift(y)}^{\j}}(k)\right)
  =\tt_V\left(\tr_r^{A_y^{\j}}\left((\Id_{V}\otimes
      f)h\right)\right)=(\theta^{\j}_y)^{\pm1}\tt_{V\otimes
    A_y^{\j}}(h)
$$ 
where $f$ is defined in Equation \eqref{E:DefOff} and $h$ is any
element of $\End_\cat(V\otimes A_y^{\j})$.  Thus, we have a well
defined $\Gr$-Markov trace and the theorem is proved.
\end{proof}

\begin{cor}
  The trace coloring system given in Theorem \ref{T:tcsys} leads to a
  well defined link invariant of $\Yd$-admissible $\Gd$-link, via
  Theorem \ref{T:GlinkInv}.
\end{cor}

\subsection{Computation of the holonomy $R$-matrix for $N=4$}
\newcommand{\matr}[4]{{\small
    \left(\begin{array}{cc}
      #1&#2\\#3&#4
    \end{array}\right)}}
\newcommand{\qi}{\ensuremath{\mathbf i}} 
\newcommand{\RP}{\ensuremath{\mathsf R}} 
\renewcommand{\P}{\ensuremath{\Phi}} 
Here we present some computations for the elementary $N=4$ case.  Then $r=2$,
$\xi=\qi=\exp(i\pi/2)$, $\s=\xi^{-\frac{r(r-1)}2}=-\qi$, $\qn1_\qi=2\qi$. Let
$\P^t=\exp(t\f)$.  Equation \eqref{E:Uwithf} implies,
$$
[E,\f]=\F{1}\qn{H-1}_\qi={2\qi}F\qn{H-1}_\qi,\quad
[E^2,\f]=\qn{H;2}_\qi!=-\qi\qn{2H}_\qi.$$
$$[E,\f^n]=n{2\qi}F\qn{H-1}_\qi\f^{n-1}\implies
[E,\P^t]=2t\qi F\qn{H-1}_\qi\P^t,$$
$$\Ad_t(E)=\P^tE\P^{-t}=E-2t\qi F\qn{H-1}_\qi.$$
Choose as in Theorem \ref{T:tcsys} a basis $(v_0,v_1)$ of the module
$V_\alpha^t=\Phi^{-t}V_\alpha$ whose degree is given by
$$\big(\kappa=-\qi^{2\alpha}=-\exp(i\pi\alpha)\quad,
\quad\ve=-\qi t\qn{2\alpha}=2t\sin(\pi\alpha)\big).$$ Then the actions of
$K,E,F$ in this basis are given by the matrices
$$K=\matr{\qi^{\alpha+1}}00{\qi^{\alpha-1}}\quad F=\matr0010\quad 
E=\matr0{\cos\frac{\pi\alpha}2}{4t\sin\frac{\pi\alpha}2}0$$ Remark that the
matrix of $E$ on $V_\alpha^t$ is equal to the matrix of $\Ad_t(E)$ on
$V_\alpha^0=V_\alpha$.

The restriction of $R_\xi:V_\alpha^{t_1}\otimes V_\beta^{t_2}\to V_\alpha^{t'_1}\otimes
V_\beta^{t'_2}$ is given by
$$R_\xi=(\Phi^{-t'_1}\otimes\Phi^{-t'_2})\,\qi^{H\otimes
  H/2}\,(\Phi^{t_1}\otimes\Phi^{t_2})\,(1+2\qi E\otimes F)$$ where
$t'_1=-\qi^{-2\beta}t_1$ and $t'_2=(1-\qi^{-4\alpha})t_1-\qi^{-2\alpha}t_2$.
In the basis ordered with the lexicographic order, we have
$$R_\xi=\qi^{\frac{\alpha\beta-1}2}\,\qi^{-\frac{\alpha+\beta}2}\RP(\qi^\alpha,\qi^\beta,t_1)\text{ where }\RP(a,b,t)=\left( \begin {array}{cccc} 
\qi ab&0&0&0
\\
0&b& \qi b(a+a^{-1})&0
\\
0&0&a&0
\\
4\qi t(a-a^{-1})&0&0&\qi
\end {array} \right)
$$
The holonomy braid relation implies that $\RP$ satisfy the following:
$$
\RP_{12}\bp{a,b,-\frac{t_1}{c^2}}\,\RP_{13}(a,c,t_1)\,\RP_{23}(b,c,t_2)=
\hspace*{25ex}\,
$$
$$
\hspace*{25ex}\RP_{23}\bp{b,c,\frac{t_1(a^4-1)-t_2a^2}{a^4}}\,
\RP_{13}\bp{a,c,-\frac{t_1}{b^2}}\,\RP_{12}(a,b,t_1)
$$
The matrices $\RP$ lift the set-theoretical Yang-Baxter map\\
\centerline{$((a,t_1),(b,t_2))\mapsto \bp{(a,-\frac{t_1}{b^2})
,(b,(1-\frac1{a^4})t_1-\frac{t_2}{a^2})}$.}\\

Computations for $N=4,6$ lead to the following conjecture~: Let $L$ be
a link whose $n$ components are colored by elements $\j_i\in I$.  Let
$\rho:\pi_1(S^3\setminus L)\to\Gd$ be a representation of its group.
The map $\matr\kappa \ve0{\kappa^{-1}}\mapsto \j_i\bp{\frac{N
    {\j}(\s^2\kappa)}{2i\pi r}}$ is constant on conjugacy class thus
its value $\alpha_i$ on any meridian of the $i$\textsuperscript{th}
component of $L$ is well defined.  Then the semi-cyclic invariant
(associated to root of unity) of $\Gd$-links is given by
$F'(L,\rho)=\mathsf{N}(L,(\alpha_1,\ldots,\alpha_n))$ where
$\mathsf{N}$ is the nilpotent (or ADO) invariant defined in \cite{GPT}
(see also \cite{ADO,CGP}) for the same root of unity.

\end{document}